\documentclass[11pt]{amsart}

\usepackage{amsmath}
\usepackage{amsfonts}
\usepackage{amssymb}
\usepackage{mathtools}
\usepackage{graphicx}
\usepackage[all,cmtip]{xy}
\usepackage{enumitem}
\usepackage{aliascnt}
\usepackage[margin=3.56cm]{geometry}
\usepackage{mathrsfs}
\usepackage{physics}

\theoremstyle{plain}
\newtheorem{theorem}{Theorem}[section]

\newtheorem{assumption}[theorem]{Assumption}

\newaliascnt{corollary}{theorem}
\newtheorem{corollary}[corollary]{Corollary}
\aliascntresetthe{corollary}

\newaliascnt{lemma}{theorem}
\newtheorem{lemma}[lemma]{Lemma}
\aliascntresetthe{lemma}
\newaliascnt{proposition}{theorem}
\newtheorem{proposition}[proposition]{Proposition}
\aliascntresetthe{proposition}

\newaliascnt{hypotheses}{theorem}

\aliascntresetthe{hypotheses}

\DeclareMathOperator{\SL}{SL}
\DeclareMathOperator{\PSL}{PSL}
\DeclareMathOperator{\SO}{SO}

\DeclareMathOperator{\Sym}{Sym}
\DeclareMathOperator{\GL}{GL}
\DeclareMathOperator{\Sp}{Sp}
\DeclareMathOperator{\Aut}{Aut}
\DeclareMathOperator{\Span}{Span}
\DeclareMathOperator{\sgn}{sgn}
\DeclareMathOperator{\Nm}{Nm}
\DeclareMathOperator{\Spin}{Spin}

\theoremstyle{definition}
\newaliascnt{definition}{theorem}
\newtheorem{definition}[definition]{Definition}
\aliascntresetthe{definition}

\newaliascnt{example}{theorem}
\newtheorem{example}[example]{Example}
\aliascntresetthe{example}

\newaliascnt{remark}{theorem}
\newtheorem{remark}[remark]{Remark}
\aliascntresetthe{remark}

\newaliascnt{remarks}{theorem}

\aliascntresetthe{remarks}

\numberwithin{equation}{section}

\usepackage[pdftex,breaklinks,colorlinks,
citecolor=blue,
linkcolor=blue,
urlcolor=blue]{hyperref}

\newcommand{\tp}[1]{\prescript{t}{}#1}
\newcommand{\nm}[1]{\left\lVert#1\right\rVert}

\begin{document}

\title[Geodesic linking numbers and genus 2 Siegel modular form]{Linking numbers and non-holomorphic Siegel modular forms}

\begin{abstract}
We study generating series encoding linking numbers between geodesics in arithmetic hyperbolic $3$-folds. We show that the series converge to functions on genus $2$ Siegel space and that certain explicit modifications have the transformation properties of genus $2$ Siegel modular forms of weight $2$. This is done by carefully analyzing the integral of the Kudla--Millson theta series over a Seifert surface with geodesic boundary. As a corollary, we deduce a polynomial bound on the linking numbers.
\end{abstract}
\author{Mads B. Christensen}
\maketitle
\setcounter{tocdepth}{1}
\tableofcontents

%--------------------------------------------------------
\section{Introduction}

In a finite volume complete hyperbolic $3$-fold $M$ there is a rigid set of closed geodesics in natural bijection with the essential non-peripheral free homotopy classes of $\pi_1 M$. Furthermore, $M$ contains non-closed geodesics connecting the finitely many cusps. A natural tempting objective is to compute linking number invariants of these geodesics, and in particular when $M$ is arithmetic one might expect such invariants to carry interesting number theoretic information. For comparison and perspective, in a different but analogous setting Duke, Imamoglu, and Toth \cite{MR3635902} related the linking numbers of modular knots in $\SL_2(\mathbb{Z})\backslash \SL_2(\mathbb{R})$ to limiting values of modular cocycles for $\SL_2(\mathbb{Z})$, which subsequently inspired a surge of fascinating developments \cite{MR4194897,MR4538040}.

We realize $M$ as a quotient $\Gamma \backslash \mathcal{H}$ where $\mathcal{H}$ is hyperbolic $3$-space and $\Gamma\subset \PSL_2(\mathbb{C})$ is an arithmetic Kleinian group. %If $\mathbb{H}$ denotes the Hamilton quaternions with the Euclidean metric $|dz|^2$ then
%\[
%    \mathcal{H} =\{z = x + jy\in \mathbb{H}: x\in \mathbb{C},y\in \mathbb{R}_{>0}\},
%\]
%with metric $|dz|^2/y^2$, and $\PSL_2(\mathbb{C})$ acts on $\mathcal{H}$ by
%\[
%    \gamma z = (az + b)(cz + d)^{ - 1},\hspace{1cm} \gamma =\begin{pmatrix} a & b \\ c & d \end{pmatrix} \in \PSL_2(\mathbb{C}),\;z\in \mathcal{H}.
%\]
To illustrate our results we consider the case where $\Gamma\subset\PSL_2(\mathbb{Z}[i])$ and where $M$ may also be realized as the complement of a link $\ell$ in $S^3$. There are infinitely many such $M$, obtained for example by taking cyclic covers of the Borromean link complement \cite{MR503003,MR1937957}. %More generally, our results apply to arithmetic Kleinian groups $\Gamma$ of simplest type, which also includes subgroups of $\PSL_2(\mathcal{O})$ for any imaginary quadratic order $\mathcal{O}$.
In the body of the paper we prove a more general result, for example with $\Gamma\subset\PSL_2(\mathcal{O})$ and $\mathcal{O}$ any imaginary quadratic order.

\subsection{Hyperbolic geodesics} 

For now, we will consider (oriented) closed geodesics $c$ for which the corresponding conjugacy class of elements $\gamma\in\Gamma$ have real trace. We follow \cite{MR1937957} and call $c$ and $\gamma$ hyperbolic. 

There is an abundance of hyperbolic geodesics in $M$. In fact, $M$ contains infinitely many totally geodesic immersed surfaces, all of which are arithmetic. It is well known that these correspond to $\Gamma$-equivalence classes of binary Hermitian forms over $\mathbb{Z}[i]$, and they are naturally parametrized by their discriminant which can be any positive integer \cite{MR1108918}. Each such arithmetic hyperbolic surface contains infinitely many closed geodesics, all of which are hyperbolic in $M$.

Supposing now that $c$ and $c'$ are two distinct hyperbolic geodesics in $M$, then %regarding $c$ and $c'$ as knots in $S^3$ 
the linking number of $c$ and $c'$ in $S^3$ is given by
\[
    \iota(c';s) = \text{the oriented intersection number of $c'$ and $s$}
\]
where $s\subset S^3$ is a Seifert surface with boundary $c$, and intersections on $c$ are counted with half-multiplicity.\footnote{A strange feature of arithmetic hyperbolic $3$-folds is that their geodesics may have intersections \cite{MR1458971}. Hence we also have to consider surfaces $s$ that do not embed in $M$, however we will loosely still call this a Seifert surface for $c$. } 

To obtain interesting results about linking numbers we sum over varying $c'$, and more precisely we sample hyperbolic geodesics from the intersections of arithmetic hyperbolic surfaces. For two distinct arithmetic hyperbolic surfaces in $M$, their intersection consists of a finite number of geodesics. Along any of these geodesics, the surfaces have a constant angle of intersection $\theta$ satisfying $(\tan\theta)^2\in \mathbb{Q}$. There is the following dichotomy
\begin{enumerate}[label = (\roman*)]
    \item if $\tan\theta\not\in \mathbb{Q}$ then the geodesic is closed and hyperbolic;
    \item if $\tan\theta\in \mathbb{Q}$ then the geodesic is of infinite length and has endpoints on the link at infinity $\ell$.
\end{enumerate}

Now, let $T=(\begin{smallmatrix}
    t_1 & t_0 \\ t_0 & t_2
\end{smallmatrix})$ be a half-integral positive definite symmetric matrix, and let $s$ be a Seifert surface in $M$. We let $\iota(T;s)$ be the sum of $\iota(c';s)$ over all intersection components $c'$ of pairs of arithmetic hyperbolic surfaces of discriminant $t_1$ and $t_2$, respectively, and angle of intersection $\theta$ satisfying $\cos\theta=t_0 / \sqrt{t_1t_2}$.\footnote{For $c'=\partial s$ we define $\iota(c';s)$ to be the oriented number of self-intersections of $c'$ plus the self-linking number with respect to the framing provided by the two arithmetic hyperbolic surfaces.
} In terms of the dichotomy above, this gives a sum over hyperbolic geodesics precisely when $\det T$ is not a rational square. %When $t_0=0$ we let $\iota(T;s)=0$. 

Due to symmetries, the resulting number $\iota(T;s)$ depends only on the $\PSL_2(\mathbb{Z})$-equivalence class of $T$, with its sign reversed under $t_0\mapsto -t_0$. Also, from the dichotomy above, $\iota(T;s)$ depends only on the boundary of $s$, when $\det T$ is not a square. When $\det T$ is a square $\iota(T;s)$ does depend on the choice of $s$, however see Remark \ref{rem:ikm}.

\begin{example}\label{exa:ife}
    Let $\ell$ be the $(-2,3,8)$-Pretzel link, and let $M$ be the complement of $\ell$. There is a known presentation $M=\Gamma\backslash \mathcal{H}$ where $\Gamma\subset \PSL_2(\mathbb{Z}[i])$ has index $12$ \cite{MR1020042}. For the hyperbolic geodesic $c$ corresponding to $\gamma= (\begin{smallmatrix}
        196 & -135 \\ -45 & 31
    \end{smallmatrix})\in \Gamma$, Table \ref{tab:nzl} shows the first few nonzero values of $\iota(T;s)$, for $\det T$ not a square.

\begin{table}[ht]
    \centering
    \begin{tabular}{ c c c | c }
        $t_1$ & $t_2$ & $t_0$ & $\iota(T;s)$ \\\hline 
        $2$ & $3$ & $1/2$ & $8$ \\
        $2$ & $4$ & $1/2$ & $24$ \\
        $2$ & $5$ & $1/2$ & $16$ \\
        $3$ & $4$ & $1$ & $2$ \\
        $2$ & $6$ & $1/2$ & $-4$ \\
        $3$ &  $4$ & $1/2$ & $-4$ \\
        $2$ & $7$ & $1/2$ & $8$ \\
        $3$ & $5$ & $1$ & $4$ \\
        $3$ & $5$ & $1/2$ & $-32$ \\
        & \vdots && \vdots
    \end{tabular}
    \caption{Nonzero values of $\iota(T;s)$, for $T$ reduced with $t_0\geq0$ and $\det T<15$ not a square, and any Seifert surface $s$ for $c$.
    }
    \label{tab:nzl}
    \end{table}
\end{example}

\subsection{Main results}

Let $\mathfrak{H}$ denote genus $2$ Siegel space and let $\Sym_2(\mathbb{Z})$ be the set of half-integral symmetric square matrices. For $\tau\in \mathfrak{H}$ and a Seifert surface $s$ we define the generating series
\begin{equation}\label{iis}
    I(\tau;s) : = \sum_{\substack{
    T\in\Sym_2(\mathbb{Z})
    \\
    \text{pos. def.}}} 
    \iota(T;s)q^T
\end{equation}
where $q^T$ is short-hand for $e^{2\pi i \tr(T\tau)}$.

Our main new construction is an explicit canonical function $G^\ast(\tau;c)$, associated with every hyperbolic geodesic $c$ of $M$, which provides a (non-holomorphic) completion of \eqref{iis} similar to those in the literature on mock modular forms.

More precisely, we let $c$ be a union of hyperbolic geodesics $c_1,\ldots ,c_r$ and assume that $c$ is nullhomologous in $M$. This allows up to pick a Seifert surface $s$ for $c$ inside $M$, which we do henceforth.

\begin{theorem}\label{thm:imt}
    The function
    \[
        F(\tau;s) = I(\tau;s) +\sum_{i = 1,\ldots,r} G^\ast(\tau;c_i)
    \]
    transforms as a Siegel modular form of genus $2$ and level $4$, that is
    \[
        F((a\tau + b)(c\tau + d)^{ - 1}) = (c\tau + d)^2F(\tau)
    \]
    for all $(\begin{smallmatrix} a & b \\ c & d \end{smallmatrix}) \in \Sp_4(\mathbb{Z})$ with $a\equiv d \equiv (\begin{smallmatrix} 1 & 0 \\ 0 & 1 \end{smallmatrix}) \bmod 4$ and $b\equiv c\equiv (\begin{smallmatrix} 0 & 0 \\ 0 & 0 \end{smallmatrix}) \bmod 4$.
\end{theorem}
\begin{remark}\label{rem:ikm}
    For $r=0$ where $c$ is empty, Theorem \ref{thm:imt} is due to Kudla and Millson \cite{MR1079646} and it states that $I(\tau;s)$ is a Siegel cusp form. In this case $s$ is a closed surface and $\iota(c';s)$ is given by the homological intersection pairing
    \[
        H_1(S^3,\ell)\times H_2(M)\to \mathbb{Z}.
    \]
    In particular, $\iota(c';s)=0$ when $c'$ is closed, and for $c'$ as in (ii) $\iota(c';s)$ equals the linking number of $s$ and the endpoints of $c'$ on $\ell$.

    For $r>0$, when $\det(T)$ is a square, this means that $\iota(T;s)$ will differ only by the coefficients of Siegel cusp forms for different choices of the Seifert surface $s$.
\end{remark}

\begin{remark}
    Several authors have considered analogues of Theorem \ref{thm:imt} for genus $1$. In \cite{MR3796415} Brunier et. al. prove that generating series of winding numbers of geodesics on the modular surface, are mock modular forms of weight $3/2$. In \cite{MR4570161} Funke and Kudla prove that generating series of linking numbers of certain cycles in the symmetric space of $\SO(2,n)$ have a non-holomorphic modular completion. 
\end{remark}

A non-trivial consequence of Theorem \ref{thm:imt} is that the series $I(\tau;s)$ converges, and consequently that $\iota(T;s)$ have sub-exponential growth in $T$. 
Using a Hecke-bound type argument together with estimates on the Fourier coefficients of $G^\ast(\tau;c)$, we are also able to obtain a polynomial estimate for $\iota(T;c)$ in $\det T$. 

\begin{theorem}\label{thm:ist}
    For every $\varepsilon>0$
    \begin{equation}\label{tmti}
        \abs{\iota(T;s)} \ll \det (T)^{3/2 + \varepsilon}.
    \end{equation}
\end{theorem}

By Siegel's theorem, \eqref{tmti} can be recast in terms of the invariants of real quadratic fields. Let $c(T)$ denote the union of the geodesics associated with $T$, hence $\iota(T;s)$ can be interpreted as the linking number of $c(T)$ and $c$. For simplicity, suppose that $-4\det T$ is a (negative) fundamental discriminant. Let $h$ and $r$ denote the class number and regulator, respectively, of the real quadratic field $\mathbb{Q}(\sqrt{\det T})$. Then there exists $m_0>0$, not depending on $T$, such that every geodesic $c'$ in $c(T)$ has length $m r$ for a positive integer $m=m(c')<m_0$. Furthermore, Theorem \ref{thm:ist} implies
\[
    \abs{\iota(T;s)} \ll (hr)^{3+\varepsilon}.
\]

It would be interesting to combine Theorem \ref{thm:ist} with information on the number of geodesics in $c(T)$, or on the total length of $c(T)$. 
We just mention some known results in this direction for analogous situations. For genus $1$, there is a known formula for the number of arithmetic hyperbolic surfaces of a given discriminant $t$ which depends only on $t\bmod 4$ \cite{MR1108918}. % $t$ from which it follows trivially that this count is $O(t^\varepsilon)$ \cite{MR1108918}. 
Rickards has also studied the analogous problem involving intersections of geodesics of fixed (fundamental) discriminants and intersection angle inside an arithmetic hyperbolic surface, and has obtained interesting formulas in terms of the prime factorization of $\det T$ \cite{MR4259279,MR4565144}. %$O(\det(T)^\varepsilon)$ bound \cite{MR4259279,MR4565144}. 

The $3/2+\varepsilon$ exponent in Theorem \ref{thm:ist} appears to be optimal with our current method; see Remark \ref{rem:eeo}.

\begin{example}\label{exa:ise}
    Continuing Example \ref{exa:ife}, we illustrate the explicitness of the function $G^\ast(\tau;c)$. First, one has a decomposition
    \[
        G^\ast(\tau;c) = \sum_{\mu \in (\mathbb{F}_{229})^2} \Theta(\tau)_{\mu} G^\ast(\tau)_{\mu}.
    \]
    The function $\Theta(\tau)_\mu$ is just a holomorphic theta function
    \[
        \Theta(\tau)_\mu = \sum_{\substack{T\in \frac{1}{229}\Sym_2(\mathbb{Z}) \\ \text{pos. def.}}} r(T)_\mu q^T,
    \]
    with coefficients the representation numbers
    \begin{align*}
        r(T)_\mu = \#\left\{N \in M_2(\mathbb{Z}) : \tp{N} 
        \begin{pmatrix}
            1 & 0 \\ 0 & 1/229
        \end{pmatrix} 
        N = T,\; N\equiv \begin{pmatrix}
            \ast & \ast 
            \\
            \mu_1 & \mu_2
        \end{pmatrix} \bmod 229\right\}.
    \end{align*}
    The second function $G^\ast(\tau)_\mu$, being responsible for the non-holomorphic contribution, is more interesting. It has a Fourier expansion
    \[
        G^\ast(\tau)_\mu = \sum_{\substack{T\in \frac{1}{229}\Sym_2(\mathbb{Z}) \\ \text{indefinite}}} \rho(T)_\mu W^\ast(\tr(Tv), 4\abs{\det(Tv)}) q^T,
    \]
    with $v=\Im\tau$ and $W^\ast$ the analytic function defined in Corollary \ref{cor:dcc}.
    Here, the $\rho(T)_\mu$ are (indefinite) representation numbers
    \[
        \rho(T)_\mu = \sum_{N\in \gamma_P\backslash \mathscr{R}_\mu} \sgn\det N,
    \]
    \[
        \mathscr{R}_\mu = \left\{N\in M_2(\mathbb{Z}):
        \tp{N} P N = T , \; N\equiv \begin{pmatrix}
            \ast & \ast 
            \\
            \mu_1 & \mu_2
        \end{pmatrix} \bmod 229\right\} ,
    \]
    for the indefinite matrix, and its automorph,
    \[
        P = \begin{pmatrix}
            17805 & 377/2 \\ 377/2 & 457/229
        \end{pmatrix},
        \hspace{1cm} 
        \gamma_P = 
        \begin{pmatrix}
            -647384 & -6855 \\ 61160175 & 647611
        \end{pmatrix}^2.
    \]
\end{example}

\subsection{Ideas of proof}

Theorem \ref{thm:imt} and Theorem \ref{thm:ist} are special cases of two more general results, which we state in Section \ref{sec:mgs}. These apply to arithmetic Kleinian groups $\Gamma\subset\PSL_2(\mathbb{C})\cong \SO_0(3,1)$ which preserve an integral quadratic lattice of signature $(3,1)$. We also work more generally with functions valued in Weil representations, which we have chosen to suppress in the introduction. 

The basic idea behind the proof is to consider the integral of theta series attached with the Kudla--Millson Schwartz form $\varphi$ over a Seifert surface $s$
\begin{equation}\label{kmi}
    \int_s \Theta(\tau,\varphi) = \sum_{T\in \Sym_2(\mathbb{Q})}\int_s \Theta(T,\varphi,v) q^T.
\end{equation}
We manage to compute all Fourier coefficients of \eqref{kmi}. To do so, our key input is a genus $2$ variant of the Millson Schwartz form $\psi$ which is related to  $\varphi$ by a certain transgression equation. We found this form $\psi$ using the formalism of Mathai and Quillen \cite{MR836726}. The transgression equation yields a decomposition of the $T$th Fourier coefficient of \eqref{kmi}
\begin{equation}\label{tfc}
    \int_s \Theta(T,\varphi,v)=\lim_{t\to\infty}\int_s \Theta(T,\varphi,vt) + \int_1^\infty \int_c\Theta(T,\psi,vt) \frac{dt}{t}.
\end{equation}
When $T$ is degenerate, it turns out that the $T$th Fourier coefficient vanishes. Hence we only have to deal with nondegenerate $T$. We analyze each of the two limits on the right-hand side of \eqref{tfc} separately. The second term we compute completely explicitly using explicit formulas for $\psi$ and its naturality properties when restricting to hyperbolic geodesics. This yields the function $G^\ast(\tau;c)$.

For the first term, the difficulties are geometric. When $T$ is positive definite, we show that this term converges to the intersection number $\iota(T;s)$. Intuitively, the forms $\Theta(T,\varphi,vt)$ are of a Gaussian shape peaking along $c(T)$, and as $t\to\infty$ they converge to the delta current of $c(T)$. Making this into an argument requires care when $c(T)$ and $c$ have components in common. One of our key observations is that $c(T)$ is equipped with a canonical framing from the intersecting surfaces. For $s$ in general position, the final formula is
\[
    \lim_{t\to\infty}\int_s \Theta(T,\varphi,vt) =\sum_i \varepsilon_i + \frac{1}{2}\sum_j \varepsilon_j + \sum_k w_k,
\]
where the first sum is over intersection points of $c(T)$ and $s$ on the interior of $s$, the second sum is over intersection points of $c(T)$ and $s$ on $c$, the third sum is over common components of $c$ and $c(T)$, $\varepsilon_i$ is the orientation number at a given intersection point, and $w_k$ is the number of times that $s$ winds around the given component with respect to the framing of $c(T)$.

Finally, when $T$ is not positive definite the first term in the right-hand side of \eqref{tfc} vanishes. Therefore we obtain a complete explicit description of all Fourier coefficients of \eqref{kmi}. 

In section \ref{sec:exa} we study Example \ref{exa:ile}. For this example $c$ is contained in a arithmetic hyperbolic surface $M'$ isomorphic to $Y_1(5)$ and the series $I(\tau;s)$ is closely related to a similar series encoding winding numbers on $M'$ between $c$ and special $0$-cycles on $M'$. The special $0$-cycles are obtained by intersecting geodesics in the same systematic manner as in \cite{MR4259279}. We describe them completely explicitly as weighted sums of the CM points in $Y_1(5)$ attached to primitive positive definite binary quadratic forms $q$. More precisely, the weights for a given $q$ is the difference between representation numbers of the square and inverse square of $q$ in the form class group. This allows us to compute coefficients of $I(\tau;s)$.

\subsection*{Acknowledgements}
This work was done during my PhD at University College London. I thank my supervisor Luis Garcia for introducing me to this area and for his many suggestions and advice throughout. A previous version of this paper, without Theorem \ref{thm:ist}, was posted online in October 2024, and I thank Jens Funke, Yingkun Li, Claudia Alfes and Jan Bruinier for their interest, questions, and comments since then. I also thank Asbjørn Nordentoft, Yuan Yang, and David Angdinata for various useful discussions.

%--------------------------------------------------------
\section{Geodesic linking numbers}\label{sec:tln}

In this section, we describe the Archimedean setup and results of the paper. We refer to \cite{MR1937957}, \cite{MR651982}, and \cite{MR553218} for further details on Kleinian groups and symmetric spaces of orthogonal groups.

%--------------------------------------------------------
\subsection{Hyperbolic \texorpdfstring{$3$}{3}-folds}\label{sec:tst}

Let $(V,Q)$ be a real quadratic vector space of signature $(3,1)$ and let $G$ be the connected component of the identity of $\SO(V)$. It is convenient to use the following explicit model for $V$
\[
    V = \left\{\begin{pmatrix} - x & y \\ z & \overline x\end{pmatrix} : x\in \mathbb{C},\;y,z\in \mathbb{R} \right\} \hspace{1cm} Q(X)= -\det X = |x|^2+yz.
\]
We denote by $(\;,\;)$ the associated bilinear form on $V$ with the normalization such that $\frac{1}{2}(X,X)=Q(X)$. By a quick calculation 
\begin{equation}\label{bfe}
    (X,Y) = \tr(X\overline{Y}) \hspace{1cm} X,Y\in V.
\end{equation}
The Clifford algebra of $V$ is $M_2(\mathbb{H})$, where $V$ is embedded by the map $X\mapsto Xj$. The even part of the Clifford algebra is $M_2(\mathbb{C})$, so in particular $\Spin V = \SL_2(\mathbb{C})$ and thus $G = \PSL_2(\mathbb{C})$ with action on $V$ given by 
\begin{equation}\label{gaf}
    g\cdot X = gX\overline g^{-1} \hspace{1cm} g\in G, X\in V.
\end{equation}

Let $\mathcal{D}$ be the Grassmannian of oriented negative lines in $V$. By choosing a positively oriented basis vector for each line we may identify $\mathcal{D}$ with the two-sheeted hyperboloid $\{x\in V : Q(x)=-1\}$. For $x\in \mathcal{D}$ we have $T_x \mathcal{D} = \{ X\in V: (x,X)=0\}$, and the restriction of $Q$ to $T_x \mathcal{D}$ is positive definite hence it induces a metric on $\mathcal{D}$. We orient $V$ by taking the basis
\begin{equation}\label{vdo}
    \begin{pmatrix}
        -1 & \\ & 1
    \end{pmatrix}, \; \begin{pmatrix}
        -i & \\ & -i
    \end{pmatrix}, \; \begin{pmatrix}
        & 1 \\ 1 &
    \end{pmatrix}, \; \begin{pmatrix}
        & 1 \\ -1 & 
    \end{pmatrix}\in V
\end{equation}
to be positively oriented. Then we orient $\mathcal{D}$ by declaring that a basis $Y_1,Y_2,Y_3\in T_x\mathcal{D}$ is positively oriented if and only if the basis $Y_1,Y_2,Y_3,x\in V$ is positively oriented. The obvious action identifies $\SO(V)$ with the full (orientation-preserving) isometry group of $\mathcal{D}$.

Each component of $\mathcal{D}$ provides a model for hyperbolic $3$-space. Sometimes, it is convenient to use the Poincare half-space model instead. 
Let $\mathbb{H}$ be the (Hamilton) quaternions, and let $|dz|^2$ be the Euclidean metric on $\mathbb{H}$. Then the half-space model for hyperbolic $3$-space is given by
\[
    \mathcal{H} =\{z = x + jy\in \mathbb{H}: x\in \mathbb{C},y\in \mathbb{R}_{>0}\},
\]
with metric $|dz|^2/y^2$, and $G$ acts by isometries on $\mathcal{H}$ by
\[
    g z = (az + b)(cz + d)^{ - 1},\hspace{1cm} g =\begin{pmatrix} a & b \\ c & d \end{pmatrix} \in G,\;z\in \mathcal{H}.
\]
It is easy to check that the map
\[
    \mathcal{H}\to \mathcal{D}\hspace{1cm} z=x+jy\mapsto X(z):=\frac{1}{y}\begin{pmatrix} -x & |z|^2 \\ -1 & \overline x \end{pmatrix}.
\]
provides a $G$-equivariant orientation-preserving isometry between $\mathcal{H}$ and one of the two connected components of $\mathcal{D}$. In particular, we also have an identification
\[
    T_z \mathcal{H} = \{X\in V : (X,X(z))=0\}.
\]
Let $\Gamma\subset G$ be a torsion-free Kleinian group, and let $M:=\Gamma \backslash \mathcal{H}$ be the associated hyperbolic $3$-manifold. In this paper, our main interest is in subgroups $\Gamma$ that stabilize an integral lattice $L\subset V$. When $\Gamma$ is of finite index in the full stabilizer group of $L$, then $\Gamma$ is an example of an arithmetic Kleinian group of Maclachlin--Reid type \cite{MR2860192,MR1937957}.

\begin{example}\label{exa:ile}
    Let
    \[
        L = \left\{\begin{pmatrix} - x & y \\ z & \overline x\end{pmatrix} : x\in \mathbb{Z}[i],\;y,z\in \mathbb{Z} \right\}.
    \]
    This is an integral lattice in $V$ with stabilizer subgroup equal to $\PSL_2(\mathbb{Z}[i])$. Various famous hyperbolic $3$-folds can be obtained from torsion-free subgroups $\Gamma\subset \PSL_2(\mathbb{Z}[i])$. For example, there are exactly two torsion-free subgrous of index 12, up to $G$-conjugacy, and the associated hyperbolic $3$-folds are the complements of the Whitehead link and the $(-2,3,8)$-Pretzel link, respectively. There are no torsion-free subgroups of index less than $12$, and in fact these two examples are the $2$-cusped hyperbolic $3$-manifolds of minimal volume \cite{MR2661571}.

    There are obvious generalizations obtained by replacing $\mathbb{Z}[i]$ with other imaginary quadratic orders. For example, the figure-eight knot complement occurs for the Eisenstein integers. We refer to \cite{MR1020042,MR503003} for many more examples.
\end{example}

%--------------------------------------------------------

\subsubsection{Geodesic cycles}

Let $U\subset V$ be a positive definite subspace of dimension $r>0$. Then we have a totally geodesic submanifold $\mathcal{D}_U\subset \mathcal{D}$ of dimension $3-r$ that parametrises the negative lines in $V$ that are orthogonal to $U$ \cite{MR1079646}. Equivalently, we may identify $\mathcal{D}_U$ with the Grassmannian of oriented negative lines in the quadratic space $U^\perp$ of signature $(3-r,1)$. In particular $T_x \mathcal{D}_U = \{X\in U^\perp : (x,X)=0\}$. If $U$ is also oriented, then we orient $U^{\perp}$ via $V=U\oplus U^{\perp}$ and $\eqref{vdo}$. Then we also obtain an orientation of $\mathcal{D}_U$ by the rule that $Y\in T_x\mathcal{D}_U$ is positively oriented if and only if the basis $Y,x\in U^{\perp}$ is positively oriented.

If $x\in \mathcal{D}_U$, then tautologically $U\subset T_x \mathcal{D}$ and $T_x\mathcal{D}_U = T_x\mathcal{D}\cap U^\perp$. It follows that there is an orthogonal decomposition 
\begin{equation}\label{tod}
    T_x\mathcal{D}=U\oplus T_x\mathcal{D}_U,
\end{equation}
which yields an identification between the normal bundle $N_{\mathcal{D}_U / \mathcal{D}}$ and the trivial bundle with fiber $U$. If $U$ is equipped with a basis, this gives a framing of $\mathcal{D}_U$. Similarly, if $U$ is equipped with an orientation, this gives an orientation of $N_{\mathcal{D}_U/\mathcal{D}}$. It is easy to check that the orientations of $N_{\mathcal{D}_U / \mathcal{D}}$, $\mathcal{D}_U$, and $\mathcal{D}$ are compatible with \eqref{tod}.

We transfer these definitions to $\mathcal{H}$. 
Let $c_U\subset \mathcal{H}$ be the component of $\mathcal{D}_U$ contained in $\mathcal{H}$. Thus
\begin{align*}
    c_U &= \{z\in \mathcal{H} : (X,X(z)) = 0\;\text{for all}\;X\in U\}, \\
    T_z c_U &= \{ 
    Z\in U^{\perp}: (Z,X(z))=0
    \}, \\
    N_{c_U/\mathcal{H},z} &= U.
\end{align*}
When $U$ is equipped with a basis, we similarly obtain a framing of $c_U$, and when $U$ is oriented we obtain compatible orientations of $N_{c_U/\mathcal{H}}$ and $c_U$.

Let $G_U\subset G$ and $\Gamma_U\subset \Gamma$ be the pointwise stabilizers of $U$. In fact, $G_U$ coincides with the connected component of the identity in $\SO(U^\perp)$, and in particular it acts by orientation-preserving isometries on $c_U$. We let $C(U):= \Gamma_U\backslash c_U$. Then we have an immersive totally geodesic map
\[
    c(U): C(U)\to M,
\]
and we also have an identification between the normal bundle $Nc(U)$ and the trivial bundle on $C(U)$ with fiber $U$. In particular if $U$ is equipped with a basis, resp. an orientation, this gives a framing of $c(U)$, resp. compatible orientations of $Nc(U)$ and $C(U)$.

The cycles $c_U$ and $c(U)$ are equivariant in the following sense. For $g\in G$ we have isomorphisms
\begin{equation}\label{cge1}
    c_{g^{-1}U}\cong c_{U},\hspace{1cm}\text{and}\hspace{1cm}G_{g^{-1}U}\cong G_U,
\end{equation}
given by multiplication and conjugation by $g$, respectively. For $\gamma\in \Gamma$ we similarly have $\Gamma_{\gamma^{-1}U}\cong \Gamma_U$ and a commutative diagram
\begin{equation}
\xymatrix{
C(\gamma^{-1}U)\ar[rr]^{\cong}\ar[rd]_{c(\gamma^{-1}U)}&& C(U)\ar[ld]^{c(U)}\\
&M&
}
\end{equation}
In particular, up to reparametrization $c(U)$ depends only on the $\Gamma$-orbit of $U$.

\begin{example}\label{exa:hpe}
    
For $r=1$, suppose that $U$ is spanned by a single vector $X=(\begin{smallmatrix} -\overline B & -C \\ A & B \end{smallmatrix})$. An easy calculation with \eqref{bfe} shows that $c_U$ is determined by the equation 
\[
    A|z|^2+ Bz+\overline{Bz}+C=0, \hspace{1cm}z\in \mathcal{H},
\]
which is a hyperbolic plane in $\mathcal{H}$. For example, if $X=(\begin{smallmatrix} i &  \\  & i \end{smallmatrix})$, then $c_U$ is determined by $\Im x=0$ and by \eqref{gaf} we see that $G_U=\PSL_2(\mathbb{R})$. 

Using \eqref{cge1} to reduce to the case $X=(\begin{smallmatrix} i &  \\  & i \end{smallmatrix})$, we see that $\Gamma_U$ is a Fuchsian subgroup of $\Gamma$ and that $c(U)$ is a totally geodesic immersed surface in $M$.
\end{example}

\begin{example}
For $r=2$, Example \ref{exa:hpe} shows that $c_U$ is the intersection of two hyperbolic planes. The condition that $U$ is positive definite then ensures that the intersection is transverse and $c_U$ is a hyperbolic line in $\mathcal{H}$. 

In this case, $G_U$ is isomorphic to $\mathbb{R}$ as an additive group, hence $\Gamma_U$ is either trivial or infinite cyclic. In the first case $c(U)$ is a geodesic of infinite length in $M$, and in the second case $c(U)$ is a closed geodesic.
\end{example}

For $r>0$, Example \ref{exa:hpe} implies that every element of $G_U$ is hyperbolic, that is $\tr g\in \mathbb{R}$. If $U$ is oriented, $r=2$, and $\Gamma_U$ is infinite cyclic, so that $c(U)$ is an oriented closed geodesic in $M$, let $\gamma_U\in\Gamma_U$ be the generator that translates on $c_U$ in the positive direction. Then $\gamma_U$ is hyperbolic and the conjugacy class of $\gamma_U$ corresponds to the free homotopy class of $c(U)$ in $\pi_1 M$. We call $c(U)$ a hyperbolic geodesic

\begin{remark}
We warn the reader that $\gamma_U$ need not be a primitive element of $\Gamma$, however, it is always primitive in the hyperbolic sense, that is, $\gamma_U\neq\gamma^n$ for any hyperbolic element $\gamma\in\Gamma$ and $n>1$. Conversely, every hyperbolic element $\gamma\in\Gamma$, which is primitive in the hyperbolic sense, is equal to $\gamma_U$ for some $U$. Unless explicitly stated otherwise, we always assume that our hyperbolic geodesics are primitive in this sense.
\end{remark}

%--------------------------------------------------------

\subsection{The forms \texorpdfstring{$\varphi^0(X)$}{f0(X)} and \texorpdfstring{$\psi^0(X)$}{y0(X)}}\label{sec:fai}

For a vector $X\in V$ we let $u(X):\mathcal{H}\to \mathbb{R}$ denote the function $z\mapsto (X,X(z))$. In coordinates, if $X=(\begin{smallmatrix} -\overline B & -C \\ A & B \end{smallmatrix})$ then we have
\begin{equation}\label{ufe}
    u(X,z) = \frac{1}{y}( A |z|^2 + Bx+\overline{Bx} + C).
\end{equation}
Rather than a single vector, we are really interested in considering pairs of vectors $X=(X_1,X_2)\in V^2$. For such a pair $X$, we let $u(X)=(u(X_1),u(X_2))$, which we regard as a function $u(X):\mathcal{H}\to \mathbb{R}^2$, and we let $r(X) = (u(X_1)^2+u(X_2)^2)^{1 /2}$ be the Euclidean norm of $u(X)$. 
\begin{definition}
    Let $X=(X_1,X_2)\in V^2$. We define two differential forms on $\mathcal{H}$ by
    \begin{align*}
        \varphi^0(X) &:= du(X_1)\wedge du(X_2) e^{ - \pi r(X)^2} \\
        \psi^0(X) &:= \frac{1}{2}(u(X_1)du(X_2) - u(X_2)du(X_1)) e^{ - \pi r(X)^2}.
    \end{align*}
\end{definition}
$\varphi^0(X)$ is closely related to the Kudla--Millson Schwartz form, see Remark \ref{rmk:kmr}. 

\begin{lemma}\label{lem:bpl} Let $X\in V^2$. Then
    \begin{enumerate}[label = (\alph*),font=\upshape]
        \item for every $g\in G$, we have $g^\ast \varphi^0(X)= \varphi^0(g^{-1}X)$ and $g^\ast\psi^0(X) = \psi^0(g^{-1}X)$;
        \item for every $h\in \SO(2)$, we have $\varphi^0(X h) = \varphi^0(X)$ and $\psi^0(Xh) = \psi^0(X)$;
        \item if $X_1$ and $X_2$ are linearly dependent, then $\varphi^0(X)=\psi^0(X)=0$;
        \item $\varphi^0(X)$ is closed;
        \item for $t>0$ we have the transgression equation
        \begin{equation}\label{fte}
            \frac{d }{dt} \varphi^0(Xt^{1 /2}) = \frac{1}{t}d\psi^0(Xt^{1 /2}).
        \end{equation}
    \end{enumerate}
\end{lemma}
\begin{proof}
(a) follows from $u(g^{-1}X_i, z)=u(X_i,g z)$. For any matrix $a\in \GL_2(\mathbb{R})$ we have $u(Xa) = u(X)a$. Since $r(X)$ is the standard Euclidean norm of $u(X)$ it follows that $r(Xa)=r(X)$ if $a\in O(2)$. If $X=(X_1,X_2)\mapsto A(X_1,X_2)$ is any alternating pairing on $V$ then $A(Xa)=\det(a) A(X)$. Applying this to the pairings $du(X_1)\wedge du(X_2)$ and $u(X_1)du(X_2)-u(X_1)du(X_1)$ then yields (b). Lastly, (d) and (e) are direct calculations using the Leibniz rule. Both sides of \eqref{fte} are equal to
\begin{equation}\label{ter}
    (1- \pi r(X)^2t) e^{ - \pi r(X)^2 t} du(X_1)\wedge du(X_2).
\end{equation}
\end{proof}

%--------------------------------------------------------

\subsubsection{Averaging}

We now construct forms on $M$ by the method of 'averaging'. In the following, we let $\sigma^0$ denote one of the symbols $\varphi^0$ or $\psi^0$. Let $\Gamma_X\subset \Gamma$ denote the stabilizer of $X$ and let 
\begin{equation}\label{asf}
    \alpha(X,\sigma^0) : = \sum_{\gamma\in \Gamma_X \backslash \Gamma} \sigma^0(\gamma^{ - 1 }X).
\end{equation}

Recall that if $M$ is a smooth manifold, then $\Omega^\ast(M)$ has a complete topology generated by the family of seminorms which measure uniform convergence of all the partial derivatives of the coefficients in compact coordinate neighborhoods. If $(\omega_i)_{i\in I}$ is a countable family of smooth forms on $M$, then we say that the series $\sum_{i\in I}\omega_i$ converges normally if $\sum_{i\in I} p(\omega_i)$ converges for all such seminorms $p$.

Before stating the next lemma, let $Q(X)$ be the symmetric matrix $\frac{1}{2}((X_i,X_j))_{i,j=1,2}$, and let
\[
    \sigma(X) := e^{-2\pi \tr Q(X)}\sigma^0(X).
\]

\begin{lemma}\label{lem:nlc}
    For any lattice $\Lambda\subset V^2$, the series 
    $
        \sum_{X\in \Lambda} \sigma(X)
    $
    converges normally in $\Omega^\ast(\mathcal{H})$. In particular, for any symmetric matrix $T\in\Sym_2(\mathbb{R})$ the series 
    \begin{equation}\label{Tsc}
        \sum_{\substack{X\in\Lambda \\ Q(X)=T}} \sigma^0(X)
    \end{equation}
    converges normally in $\Omega^\ast(\mathcal{H})$.
\end{lemma}
\begin{proof}
Let
    \[
        (X,Y)_z = (X,Y) +u(X,z)u(Y,z),\hspace{1cm} X,Y\in V.
    \]
    An easy computation shows that $(\;,\;)_z$ is the Siegel majorant associated to $X(z)\in \mathcal{D}$. For $X=(X_1,X_2)$ let $Q_z(X)$ denote the matrix $\frac{1}{2}((X_i,X_j)_z)_{i,j=1,2}$. Then we have $\tr Q_z(X) = \tr Q(X)+\frac{1}{2}r(X,z)^2$ and hence 
    \[
        \sigma(X,z) = \sigma_0(X)e^{ - 2\pi\tr Q_z(X)},
    \]
where
\begin{align*}
    \varphi_0(X)= du(X_1)\wedge du(X_2) \hspace{1cm}
    \psi_0(X)= \frac{1}{2}(u(X_1)du(X_2) - u(X_2)du(X_1)).
\end{align*}
In both cases, we see by differentiating \eqref{ufe} that $\sigma_0(X,z)$ has coefficients which are polynomial functions of $X$. Further, if $D:\Omega^\ast(\mathcal{H})\to C^\infty(\mathcal{H})$ is a differential operator which takes partial derivatives of one of the coefficients, then $D\sigma(X,z) = \sigma_{0,D}(X,z)e^{-2\pi\tr Q_z(X)}$ where $\sigma_{0,D}(X,z)$ is also polynomial in $X$.

Now let $K\subset \mathcal{H}$ be compact. We may find a positive definite quadratic form $q:V^2\to \mathbb{R}_{\geq0}$ such that $q \leq \tr Q_z$ for all $z\in K$. Then in addition, we may find a constant $C=C(q)>0$ such that $|\sigma_{0,D}(X,z)|\leq C e^{-\pi q(X)}$ for all $z\in K$. This yields
\[
    \sup_{z\in K} |D\sigma(X,z)| \leq C e^{ - \pi q(X)},
\]
and the lemma follows by comparison with the theta series of $q$.
\end{proof}

In view of Lemma \ref{lem:bpl}(c) and Lemma \ref{lem:nlc} we impose the following condition on $X$.

\begin{assumption}\label{ass:agx}
    $X_1$ and $X_2$ are linearly independent and the orbit $\Gamma X\subset V^2$ is contained in a lattice.
\end{assumption}

By Lemma \ref{lem:nlc}, we thus obtain.

\begin{corollary}\label{cor:coa}
    \eqref{asf} converges normally in $\Omega^\ast(\mathcal{H})$.
\end{corollary}

\begin{proof}
    Since $Q(\gamma^{-1}X)=Q(X)$, this follows from the convergence of \eqref{Tsc}.
\end{proof}

It follows that $\alpha(X,\varphi^0)$ and  $\alpha(X,\psi^0)$ are smooth $\Gamma$-invariant forms on $\mathcal{H}$, and therefore we may identify them with smooth forms on $M$. In fact, if $\Gamma X$ is contained in the lattice $\Lambda$ then for any $a\in \GL_2(\mathbb{R})$ the orbit $\Gamma Xa$ is contained in $\Lambda a$ and hence $Xa$ also satisfies \ref{ass:agx}. In particular $\alpha(Xa,\sigma^0)$, for $\sigma^0\in \{\varphi^0,\psi^0\}$, also define smooth forms on $M$.

\begin{remark}
    The formulas for $\varphi^0(X)$ and $\psi^0(X)$ were obtained using the formalism developed by Mathai and Quillen \cite{MR836726}. Given an oriented vector bundle $\mathscr{E}$ or rank $r$ equipped with a metric and a compatible connection, they construct natural rapidly decreasing differential forms $U_{MQ}\in \Omega^r(\mathscr{E})$ and $V_{MQ}\in \Omega^{r-1}(\mathscr{E})$. Their main property is that $U_{MQ}$ is a Thom form for $\mathscr{E}$ and a transgression equation generalizing \eqref{fte}. 
    
    In more detail, let $\mathscr{V}$ be the trivial vector bundle on $\mathcal{D}$ with fiber $V$ and let $\mathscr{L}\subset \mathscr{V}$ be the tautological oriented line bundle for which the fiber at $z$ is the line in $V$ corresponding to $z$. The bundle $\mathscr{L}$ has a natural metric and compatible connection vaguely given by projecting $Q$ and $d$ from $\mathscr{V}$. The pair of vectors $X\in V^2$ also gives a section of $\mathscr{V}^{\oplus 2}$ and projects to a section $s_X\in \Gamma(\mathscr{L}^{\oplus 2})$. Then our forms $\varphi^0(X)$ and $\psi^0(X)$ are given by pulling back $U_{MQ}\in \Omega^2(\mathscr{L}^{\oplus 2})$ and $V_{MQ}\in \Omega^1(\mathscr{L}^{\oplus 2})$ by $s_X$.
    Note that we have a trivialization of $\mathscr{L}$ and $u(X)$ is just the coordinates of $s_{X}$ in this trivialization.
\end{remark}
\begin{remark}\label{rmk:kmr}
    By the previous remark and \cite{branchereau2023kudla} we have $2^{-1 /2}\varphi(X) = \varphi_{KM}(X)$, where $\varphi_{KM}$ denotes the Kudla--Millson Schwartz form; see e.g. \cite{MR842618}. See also \cite{MR3783423} for very similar results for Hermitian symmetric domains, which give a construction of the Kudla--Millson form in terms of superconnections via \cite{MR790678}.
\end{remark}

%--------------------------------------------------------

\subsubsection{The framed geodesics $c(X)$}

Let $c_X\subset \mathcal{H}$ be the zero locus of $u(X)$. The subset $c_X$ is closely related to the geodesic submanifolds $c_U$ from the previous subsection. This is the content of Lemma \ref{lem:gcl} below. The subset $c_X$ is also closely related to the two forms $\varphi^0(X)$ and $\psi^0(X)$, since when $c_X\neq\varnothing$ they have a Gaussian shape with peaks along $c_X$.

\begin{lemma}\label{lem:gcl}
If $\Span X$ is positive definite, then $c_X=c_{\Span X}$, and otherwise $c_X=\emptyset$.
\end{lemma}
\begin{proof}
When $\Span X$ is a positive definite subspace, then it is obvious from the definitions that $c_X=c_{\Span X}$.

Thus, we must show that $u(X)\neq0$ on all of $\mathcal{H}$ when $\Span X$ is not positive definite. Equivalently, we must show that $(\Span X)^\perp$ does not contain any negative vectors when $\Span X$ is not positive definite. However, if $\Span X$ contained a nonpositive vector and its orthogonal complement contained a negative vector, then they would span a $2$ dimensional negative semidefinite subspace of $V$. This is impossible by the signature condition on $V$.
\end{proof}

Let $C(X) = \Gamma_X\backslash c_X$ and let $c(X):C(X)\to M$ denote the canonical map. By Lemma \ref{lem:gcl} we have two cases
\begin{enumerate}
    \item[-] $c(X)$ is a geodesic in $M$.
    \item[-] $c(X)$ is the empty map.
\end{enumerate}
The first case occurs if and only if $Q(X)$ is positive definite, and in this case we have a canonical identification between $Nc(X)$ and the trivial bundle with fiber $\Span X$. In particular $X_1$ and $X_2$ give a canonical framing of $c(X)$. At the same time, as we have defined $c_X$ as the zero locus of $u(X)$, we also have a canonical framing provided by the differential $du(X):Nc_X\to \mathbb{R}^2$. This yields another framing of $Nc(X)$ provided by the isomorphism
\[
    (-,X):\Span X \cong \mathbb{R}^2.
\]
%since this matches the framing provided by the differential $du(X):Nc_X\to \mathbb{R}^2$.
The two framings differ by the change of basis matrix $2Q(X)$, hence they are essentially equivalent.
\begin{lemma}\label{lem:pcx}
    The map $c(X)$ is proper.
\end{lemma}
\begin{proof}
There is only something to show if $Q(X)$ is positive definite and $\Gamma_X=1$. Then this boils down to showing that if $K\subset \mathcal{H}$ is compact, then $K\cap c_{\gamma^{-1}X}=\emptyset$ for all but finitely many $\gamma\in \Gamma$. Choose a positive definite quadratic form $q:V^2\to \mathbb{R}_{\geq0}$ such that $q\leq  \tr Q_z$ for $z\in K$, as in the proof of Lemma \ref{lem:nlc}. If $Y\in \Gamma X$ and $z\in K\cap c_Y$, then $\tr Q_z(Y) = \tr Q(Y) = \tr Q(X)$. Hence, if there were infinitely many $Y\in \Gamma X$ such that $K\cap c_Y\neq \emptyset$, then we would obtain infinitely many $Y\in \Gamma X$ with $q(Y) \leq \tr Q(X)$. This would contradict \ref{ass:agx}.
\end{proof}

\begin{remark}\label{rem:asa}
Let $\theta\in \mathbb{R}/2\pi\mathbb{Z}$ be the angle between the hyperbolic planes in $\mathcal{H}$ defined by $X_1$ and $X_2$. Evidently, $\theta$ is also the angle of intersection of the totally geodesic immersed surfaces defined by $X_1$ and $X_2$ along $c(X)$. Since $Q(X)$ coincides with the Gram matrix of $X_1$ and $X_2$ in the normal bundle of $c(X)$, we have the formula
\begin{equation}\label{iaf}
   \cos \theta = \frac{\frac{1}{2}(X_1,X_2)}{ \sqrt{Q(X_1)Q(X_2)}}
\end{equation}
\end{remark}

%--------------------------------------------------------

\subsection{The linking numbers $\iota(X;s)$}\label{sec:scl}

For a surface $s$ in $M$, our goal is to relate the limit 
\begin{equation}\label{lnd}
    \lim_{t\to \infty} \int_s \alpha(X t^{1 /2},\varphi^0)
\end{equation}
to the intersection of $s$ and $c(X)$. This leads to certain intersection numbers $\iota(X;s)$, which we now define.

One of the main properties of $\alpha(X,\varphi^0)$, following from Remark \ref{rmk:kmr} and the work of Kudla and Millson, is that it is a Poincare dual form of $c(X)$. In particular, when $s$ is closed $\int_s \alpha(Xt^{1/2},\varphi^0)$ is constant and equals the homological intersection pairing of $s$ and $c(X)$. The main difficulty therefore arises when $s$ has a nonempty boundary.

We first clarify what we mean by intersections of maps. Suppose that $a:A\to M$ is a smooth map from an oriented surface and $b:B\to M$ is an immersed curve with a fixed orientation of its normal bundle $Nb$. By an intersection point of $a$ and $b$ we mean a pair of points $x\in A$ and $t\in B$ such that $a(x)=b(t)$. We say that $a$ and $b$ are transverse at an intersection point $(x,t)$ if the map
\[
    da:T_x A\to N_{t} b
\]
is bijective, in which case we define the orientation number at $(x,t)$ to be $1$, resp. $-1$ if the map is orientation-preserving, resp. orientation reversing.

Now fix a smooth map $s:S\to M$ where $S$ is an oriented compact surface. Let $C=\partial S$ and let $c=\partial s$. We assume that $c$ is a finite union of hyperbolic geodesics $c_i:C_i\to M$ indexed by $i\in I$, that is, $c_i = c(U_i)$ where $U_i\subset V$ is a positive definite subspace of rank $2$. By Lemma \ref{lem:eog} below such a map $s$ exists, for a given $c$, if and only if $[c]\in H_1(M)$ is trivial. We abuse notation slightly and call $s$ a Seifert surface for $c$, even though $c$ may not be simple and $s$ may not even be immersed.

\begin{lemma}\label{lem:eog}
    Let $M$ be the interior of an oriented compact $3$-fold, and let $c:C\to M$ be an immersion, where $C$ is an oriented closed curve. Then there exists an oriented compact surface $S$ and a smooth map $s:S\to M$ with boundary $c$ if and only if the fundamental class $[c]\in H_1(M)$ is trivial.
\end{lemma}
\begin{proof}
When $c$ is an embedding, then this follows from a generalization of Seifert's theorem which also states that $s$ may be chosen to be an embedding \cite[4.1.3]{MR1479639}. 

In the general case, we consider a small perturbation $c'$ of $c$ which is an embedding. Then there exists an embedded surface $s'$ in $M$ with boundary $c'$, and we can perturb $s'$ backwards to obtain a map $s$ with boundary $c$.
\end{proof}

Assume $Q(X)$ is positive definite. We let%\footnote{When we write $c_i=c(X)$, we mean that there exists a commutative diagram as in \eqref{cge1} and we identify $C_i=C(X)$ via the given isomorphism. }
\[
    I_1 : = \{i\in I: c_i \neq c(X)\}, \hspace{1cm} I_2 : = \{i\in I:c_i=c(X)\}.
\]
If $(x,t)$ is an intersection point of $s$ and $c(X)$ we shall call it an inner point of $s$ if $x\not\in C$, and otherwise we call it a boundary point. Further, we call a boundary intersection point $(x,t)$ bad if $x\in C_i$ for $i\in I_2$ and $x=t$, and otherwise we call it good.

\begin{definition}
    We define the following transversality conditions:
    \begin{enumerate}
        \item[($t^0$)] $s$ and $c(X)$ are transverse at every inner intersection point.
        \item[($t_i$)] $s$ and $c(X)$ are transverse at all good boundary intersection points $(x,t)$ with $x\in C_i$, and if $i\in I_2$ then $s$ is immersive on $C_i$.
        \item[($t$)] $s$ satisfies $(t^0)$ and ($t_i$) for all $i\in I$.
    \end{enumerate}
\end{definition}

Since distinct hyperbolic lines in $\mathcal{H}$ cannot have overlapping tangent lines, it follows that a boundary intersection point $(x,t)$ is good if and only if the map
\[
    T_x C\oplus T_t C(X)\to T_{c(x)}M
\]
is injective. In particular, $s$ and $c(X)$ can never be transverse at bad boundary intersection points, whereas for good boundary intersection points the obstruction to being transverse is that the image of non-zero normal vectors at $x$ under $ds$ should not be tangent to $c(X)$. %We use this observation to prove the following lemma.

\begin{lemma}\label{lem:hti}
    There exists a homotopy of $s$ relative to $c$ which makes $(t)$ hold.
\end{lemma}
\begin{proof}
By Thom's transversality \cite{guillemin2010differential} one can find a homotopy which makes $(t^0)$ hold. Now assuming $(t^0)$, let $A=C\times [0,1)\subset S$ be a collar neighborhood of $C$, such that $s(A)$ does not meet $c(X)$. After potentially shrinking $A$, we may factor $s$ as
\[
    A\xrightarrow{b} Nc \to M
\]
where $b$ is a fiber bundle map over $C$, and $Nc\to M$ is immersive. Let $\rho:A\to \mathbb{R}_{\geq0}$ be a smooth map such that $\frac{\partial \rho}{\partial t}(x,0)\neq0$ for $x\in C$ and $\rho(x,t)=0$ for $t>\frac{1}{2}$. Then by applying Thom's transversality to the partial derivative in $t$, we find a section $n$ of $Nc$ such that the map 
\[
    b_1:A\to Nc \hspace{1cm} (x,t)\mapsto b(x,t) + \rho(x,t)n(x).
\]
is immersive on $C$. Similarly, considering the section $tn$ for $t\in [0,1]$ we obtain maps $b_t:A\to Nc$ with $b_0=b$. We then consider the homotopy $s_t$ given by the composite $A\xrightarrow{b_t} Nc\to M$. The resulting map $s_1$ is then immersive on $c$ and satisfies $(t^0)$.

Now the last step is to find a homotopy such that $s$ and $c(X)$ are transverse at all good boundary intersection points. Since $s$ is immersive on $c$, this can be achieved by doing a small rotation around $c$ near any intersection point, such that the normal direction misses the tangent direction(s) of $c(X)$.
\end{proof}

The condition that $s$ is immersive on $C_i$ when $i\in I_2$, is needed to define certain winding numbers, as we now describe. By the assumption, the differential $ds$ induces an injective map of normal bundles
\[
    ds:N_{C_i /S} \to N c_i.
\]
The bundle $N_{C_i /S}$ has a canonical orientation, where positive vectors are given by inwards pointing tangent vectors. Any two positive sections are thus connected by a straight-line homotopy, and by applying $ds$ we obtain a canonical section of $Nc_i$ with no zeroes, up to homotopy. Using the framing of $c_i=c(X)$ we obtain a map $C_i\to \mathbb{R}^2$ with no zeroes, and hence a map
\begin{equation}\label{wnm}
    f_i:C_i\to S^1
\end{equation}
which is defined uniquely up to homotopy. The degree of $f_i$ measures how many times that $s$ winds around $c_i$ with respect to the framing of $c(X)$. 

\begin{definition}
    With notation as above: \begin{enumerate}
        \item[-] When $(t^0)$ holds, we let $\iota^0(X;s)$ be the sum of the orientation numbers of $c(X)$ and $s$ at all inner intersection points.
        \item[-] When $i\in I_1$ and $(t_i)$ holds, we let $\iota_i(X;s)$ be half the sum of the orientation numbers at all boundary intersection points on $c_i$.
        \item[-] When $i\in I_2$ and $(t_i)$ holds, we let $\iota_i(X;s)$ be the degree of $f_i$ plus half the sum of the orientation numbers at all good boundary intersection points on $c_i$.
    \end{enumerate}
    Finally, when $(t)$ holds, we let
    \[
        \iota(X;s) := \iota^0(X;s)+\sum_{i\in I} \iota_i(X;s).
    \]
\end{definition}

By the compactness of $S$ and Lemma \ref{lem:pcx}, there are only finitely many inner intersection points and good boundary intersection points.

When $Q(X)$ is not positive definite, we declare that $(t)$ always holds, and we let $\iota(X;s)=0$. Thus, we have defined $\iota(X;s)$ for all $X$ and all Seifert surfaces $s$ that satisfy $(t)$. In the next subsection, we show that the definition depends only on the homotopy class of $s$ relative to $c$.

\begin{proposition}\label{prp:ihi}
Let $c$ be a union of hyperbolic geodesics in $M$ with trivial fundamental class. Let $s_1$ and $s_2$ be Seifert surfaces for $c$ that satisfy $(t)$ and are homotopic relative to $c$. Then
\[
    \iota(X;s_1)=\iota(X;s_2).
\]
\end{proposition}
Lemma \ref{lem:hti} and Proposition \ref{prp:ihi} imply the following corollary:
\begin{corollary}\label{cor:uei}
    There is a unique extension of $\iota(X;s)$ to all Seifert surfaces $s$ for $c$ which depends only on the homotopy class of $s$ relative to $c$.
\end{corollary}
We record the following lemma, which will be used in Section \ref{sec:mgs}.
\begin{lemma}\label{lem:xai}
    For all $a\in \GL_2^+(\mathbb{R})$, we have $\iota(Xa;s)=\iota(X;s)$.
\end{lemma}
\begin{proof}
We may assume that $(t)$ holds. The geodesics $c(Xa)$ and $c(X)$ differ only in their framing, which differ by the transformation $a$. In particular, since $\det a>0$ they have the same orientation, so $\iota^0(Xa;s)=\iota(X;s)$ and $\iota_i(Xa;s)=\iota_i(X;s)$ for $i\in I_1$, and the two maps \eqref{wnm} are homotopic, so $\iota_i(Xa;s)=\iota_i(X;s)$ for $i\in I_2$.
\end{proof}

%--------------------------------------------------------

\subsection{The main identity}

As in the previous section, we assume that $X$ satisfies Assumption \ref{ass:agx} and that $s$ is a Seifert surface for a union of hyperbolic geodesics. We now prove Proposition \ref{prp:ihi}, together with the main result of this section.

\begin{proposition}\label{prp:bmt}
    The limit in \eqref{lnd} exists and
    \begin{equation}\label{fsmi}
        \lim_{t\to \infty} \int_s \alpha(X t^{1 /2},\varphi^0)=\iota(X;s).
    \end{equation}
\end{proposition}

This boils down to proving \eqref{fsmi} when condition $(t)$ holds. Indeed, as $\alpha(Xt^{1/2},\varphi^0)$ is closed by \ref{lem:bpl}(d), the left-hand side of \eqref{fsmi} depends only on the homotopy class of $s$ relative to $c$. Therefore, if we prove \eqref{fsmi} when $(t)$ holds, then this proves Proposition \ref{prp:ihi}, and it follows that \eqref{fsmi} holds for any $s$ by homotoping to a Seifert surface that satisfies $(t)$.

To prove \eqref{fsmi} when $(t)$ holds, the rough idea is to analyse the limit integral on small neighborhoods around each intersection point, and show that inner intersection points contribute $1$ to the limit and that good boundary intersection points contribute $\frac{1}{2}$. As one might expect, the most difficult part lies in extracting the winding number contribution from the bad boundary intersection points, since these are not discrete. 
\begin{remark}
    In fact, $\iota(X;s)$ depends only on the homology class of $s$ relative to $c$. %This is somewhat subtle; for example the same statement is not true for any of the individual contributions $\iota^0(X;s)$ or $\iota_i(X;s)$.
    More precisely, consider the exact sequence
    \[
        0\to H_2(M)\to H_2(M,c)\to H_1(C),
    \]
    where $H_2(M,c)$ denotes relative homology of the map $c$. A Seifert surface $s$ for $c$ determines a preimage $[s]\in H_2(M,c)$ of the fundamental class $[c]\in H_1(C)$. For two different Seifert surfaces $s_1$ and $s_2$, we have $[s_1]-[s_2]=[s_3]$ for a closed surface $s_3\subset M$, by the exactness. As the forms $\alpha(Xt^{1/2},\varphi^0)$ are closed, this implies
    \[
        \iota(X;s_1) - \iota(X;s_2) = \iota(X;s_3).
    \]
    In particular, if $[s_3]=0$ then $\iota(X;s_3)=0$ and hence $\iota(X;s_1) = \iota(X;s_2)$.
\end{remark}

\begin{remark}\label{rem:sgl}
    Our proofs also work if some boundary components of $s$ are not geodesics but instead are required to be disjoint with $c(X)$. When $M$ is the complement of a link $\ell$ in a closed $3$-fold $M'$ and $c(X)$ is closed, this allows the definition of $\iota(X;s)$ to be extended to any Seifert surface $s$ in $M'$ for which $\partial s$ is a union of hyperbolic geodesics in $M$. Specifically, let $N$ be a tubular neighborhood of $\ell$ which does not intersect $c(X)$. We may homotope $s$, inside $N$, such that it is transverse to $\ell$ and remove a disk around every intersection point with $\ell$, obtaining a Seifert surface $s'$ in $M$. One can show that the resulting intersection number $\iota(X;s')$ depends only on the class $[s]\in H_2(M';c)$.
\end{remark}

\subsubsection{Unfolding}
We found it simpler to analyze the different contributions by unfolding the integral of $\alpha(Xt^{1/2})$ over $s$ to an integral of $\varphi^0(Xt^{1/2})$. 

Let $p_0:\mathcal{H}\to M$ and $p:\Gamma_X\backslash \mathcal{H}\to M$ denote the canonical covering maps. For a map $a:A\to M$ let $p^\ast A =A\times_M (\Gamma_X \backslash \mathcal{H})$ denote the fibre product and let $p^\ast a:\pi^\ast A\to \Gamma_X\backslash\mathcal{H}$ denote the canonical map. If $A$ is a manifold with boundary $B=\partial A$, and $b=\partial a$, then $p^\ast A$ is a manifold with boundary $p^\ast B$ and we have $\partial (p^\ast a) = p^\ast b$. Replacing $\Gamma_X\backslash \mathcal{H}$ with $\mathcal{H}$ we use the analogous notation for $p_0$. We have covering maps $p_0^\ast A\to A$ and $p_0^\ast A\to p^\ast A$ with Galois groups $\Gamma_X$ and $\Gamma$.

Since $u(X)$ is $\Gamma_X$-invariant it defines a function $\Gamma_X\backslash \mathcal{H}\to \mathbb{R}^2$. We denote the zero set of this function by $\tilde c_X\subset \Gamma_X\backslash \mathcal{H}$. Note that $\tilde c_X$ can be identified with $C(X)$ under the obvious map. Moreover the intersection points of $s$ and $c(X)$ naturally form a subset of (the non-compact surface) $p^\ast S$ which coincides with $\tilde s^{-1}(\tilde c_X)$, where $\tilde s:=p^\ast s$. 

Noting that $\varphi^0(Xt^{1 /2})$ is $\Gamma_X$-invariant, we identify it with a form on $\Gamma_X\backslash \mathcal{H}$. 

\begin{lemma}\label{lem:ufl2} The form $\tilde s^\ast\varphi^0(Xt^{1 /2})$  is integrable on $p^\ast S$ and 
    \[
        \int_{S} s^\ast \alpha(Xt^{1 /2},\varphi^0) = \int_{p^\ast S} \tilde s^\ast\varphi^0(Xt^{1 /2}).
    \]
    In addition, if $A\subset p^\ast S$ is a closed subset which does not contain any intersection points of $s$ and $c(X)$, then
\begin{equation}
    \lim_{t\to\infty}\int_A \tilde s^\ast \varphi^0(X t^{1 /2}) = 0.
\end{equation} 
\end{lemma}
\begin{proof}
Let $\bar s=p_0^\ast s$. If $N\subset S$ is a simply connected compact coordinate neighborhood, then $p_0^\ast N = \bigcup_{\gamma\in \Gamma}\gamma \Bar{N}$ for some connected component $\Bar{N} \subset p_0^\ast N$. Let $\rho:S\to \mathbb{R}_{\geq0}$ be a smooth function supported on such a neighborhood $N$, and let $\tilde\rho$ and $\bar \rho$ denote the composite maps $p^\ast S\to S\xrightarrow{\rho} \mathbb{R}_{\geq0}$ and $p_0^\ast S\to S\xrightarrow{\rho} \mathbb{R}_{\geq0}$, respectively. 

Let $\Phi_{\infty}\subset \Gamma$ be a set of coset representatives for $\Gamma_X \backslash \Gamma$, and let $D= \bigcup_{\gamma\in \Phi_\infty} \gamma\bar N$. We claim that
\begin{equation}\label{uf1}
    \int_S \rho s^\ast \alpha(Xt^{1 /2},\varphi^0) = \int_{D} \bar{\rho}\bar{s}^\ast \varphi^0(Xt^{1 /2})
\end{equation}
and
\begin{equation}\label{uf2}
    \lim_{t\to \infty} \int_{\bar{A}} \bar{\rho}\bar{s}^\ast \varphi^0(Xt^{1 /2}) = 0
\end{equation}
where $\bar A\subset D$ is any closed subset such that $\bar s(\bar A)$ does not meet $c_X$. 

Since $D$ maps diffeomorphically onto $p^\ast N$, it follows that
\[
    \int_{D} \bar{\rho}\bar{s}^\ast \varphi^0(Xt^{1 /2}) = \int_{p^\ast S} \tilde{\rho}\tilde{s}^\ast \varphi^0(Xt^{1 /2}),
\]
and if picking $\bar A$ to be the preimage in $D$ of $A$ we also have
\[
    \int_{\bar{A}} \bar{\rho}\bar{s}^\ast \varphi^0(Xt^{1 /2}) = \int_{A} \tilde{\rho}\tilde{s}^\ast \varphi^0(Xt^{1 /2}).
\]
Hence the lemma now follows by a partition of unity argument.

Let $x=(x_1,x_2)$ be coordinates on $\overline N$. Then
\begin{equation}\label{pff}
    \tilde s^\ast \varphi^0(X) = e^{ - \pi r(X,\tilde s(x))^2}f(X,x) dx_1\wedge dx_2,
\end{equation}
where $f(X,x)$ is a homogenous quadratic polynomial in $X$ and smooth in $x$.

For any finite subset $\Phi\subset \Phi_\infty$ we have
\begin{equation}\label{puf}
    \int_{\widetilde N} \bar \rho\bar s^\ast \sum_{\gamma\in \Phi}\varphi^0(\gamma^{ - 1 }X t^{1 /2}) = \sum_{\gamma\in\Phi}\int_{\gamma\widetilde{N}}\bar \rho\bar s^\ast \varphi^0(Xt^{ 1 /2}).
\end{equation}
The same holds if replacing $f(X,x)$ by $|f(X,x)|$. Now considering a sequence of finite subsets $\Phi_1\subset \Phi_2\subset\cdots$ with union $\Phi_\infty$, then it follows that $\bar\rho\bar s^\ast\varphi^0(X)$ is integrable from the same proof as for Corollary \ref{cor:coa}, and we deduce \eqref{uf1} since $\sum_{\gamma\in \Phi_n}\varphi^0(\gamma^{-1}Xt^{1 /2})$ converges to $\alpha(Xt^{ 1/ 2},\varphi^0)$ as $n\to\infty$ and 
\begin{equation}
    \int_{N} \rho s^\ast \alpha(Xt^{1 /2},\varphi^0) = \int_{\overline N} \bar \rho s^\ast \alpha(Xt^{1 /2},\varphi^0).
\end{equation}
In view of \eqref{pff} and \eqref{puf}, to show \eqref{uf2} it suffices to prove that
\begin{equation}\label{dcl}
    \lim_{t\to\infty}\sum_{\gamma\in \Phi_\infty} \sup_{x\in \overline N\cap \gamma^{ - 1}A} ( e^{ - \pi t r(\gamma^{ - 1}X,\bar s(x))^2}|f(\gamma^{ - 1}X t^{1 /2},x)|) =0.
\end{equation}
To this end, we will apply dominated convergence.

For any $\gamma$ we have $r(\gamma^{-1}Xt^{1 /2}\bar s(x))>0$ for $x\in \overline N\cap \gamma^{-1}A$ by assumption on $A$, so by compactness we can find a lower bound $r({\gamma})>0$. It follows that there exists a constant $C(\gamma)>0$ such that
\[
    e^{ - \pi t r(\gamma^{ - 1}X,\bar s(x))^2}|f(\gamma^{ - 1}X t^{1 /2},x)| < C(\gamma) t e^{ - r(\gamma) t}
\]
which in particular converges to $0$ for $t\to \infty$. It remains to find a summable function of $\gamma\in \Phi_\infty$ which dominates the supremum in \eqref{dcl}.

Let $q:V^2\to \mathbb{R}_{\geq0}$ be a positive definite quadratic form with $q<\tr Q_{\bar s(x)}$ for $x\in \overline N$. We have
\begin{equation}\label{qer}
    r(\gamma^{ - 1}X)^2 \geq 2(q(\gamma^{ - 1}X) - \tr Q(X)).
\end{equation}
There exists $n>0$ such that $q(\gamma^{-1}X)>2\tr Q(X)$ for $\gamma\in \Phi_\infty - \Phi_n$, hence  there exists $r_0>0$ such that
\[
    r(\gamma^{ - 1}X,\bar s(x)) > \max( q(\gamma^{ - 1}X)^{1 /2},r_0)
\]
for all $\gamma\in \Phi_\infty$ and $x\in \overline N\cap \gamma^{-1}A$. The function $t e^{-\pi t r(\gamma^{-1}X,\bar s(x))^2}$ is decreasing for $t> 1 / (\pi r(\gamma^{-1}X),\bar s(x))^2$ hence for $t> t_0 := 1 / ( \pi r_0^2)$ we have
\[
    t e^{-\pi t r(\gamma^{-1}X,\bar s(x))^2} < t_0 e^{ - \pi t_0 q(\gamma^{ - 1}X)}.
\] 
This yields the estimate
\[
    e^{ - \pi t r(\gamma^{ - 1}X,\bar s(x))^2}|f(\gamma^{ - 1}X t^{1 /2},x)| < C_0 e^{ - q_0(\gamma^{ - 1}X)}\hspace{1cm} t > t_0.
\]
for some constant $C_0>0$ and $q_0$ a positive definite quadratic form with $q_0< \pi t_0 q$. Thus the series in \eqref{dcl} is dominated by a theta series.
\end{proof}

Let $U\subset V$ be a positive definite subspace of rank $2$. 

\begin{lemma}\label{lem:ll1}
    Suppose that $c_U\cap c_X\neq \varnothing$ and $U\neq \Span X$. Then $u(X)$ restricts to a diffeomorphism of $c_U$ onto a line in $\mathbb{R}^2$ passing through the origin.
\end{lemma}
\begin{proof}
Let $z$ be the intersection point of $c_U$ and $c_X$. Then $X(z)^\perp = U+\Span X$. Since $U\neq\Span X$ this implies that $U\cap \Span X$ is one dimensional. Let $Y=y_1X_1+y_2X_2$ be a vector spanning $U\cap \Span X$. Since $u(Y)=y_1u(X_1)+y_2u(X_2)$ vanishes on $c_U$, it follows that $u(X)$ maps $c_X$ into the line
\[
    \{ u\in \mathbb{R}^2 : y_1u_1 + y_2u_2 = 0\}.
\]
Now $U^\perp$ is a hyperbolic plane, and $X(z)$ is a negative vector in it. Hence we can find $Y'\in U^\perp$ such that $Q(Y')=1$ and $(Y',X(z))=0$. Then we may find a parametrization $z(t)$, $t\in \mathbb{R}$ of $c_U$ such that 
\[
    X(z(t)) =\cosh(t) X(z) +\sinh(t) Y' \hspace{1cm} t\in \mathbb{R}.
\]
In particular $u(X,z(t)) = \sinh(t) (Y',X)$, and as $\sinh$ is a diffeomorphism of $\mathbb{R}$ this proves the lemma.
\end{proof}

For $\gamma\in \Gamma$ we have a smooth map
\begin{equation}
    a_{\gamma} :c_{\gamma^{ - 1}U}\to p^\ast C(U) \hspace{1cm} t \mapsto (\Gamma_U\gamma t,\Gamma_X t).
\end{equation}
If $\gamma'\in \Gamma_U$ and $\gamma_X\in \Gamma_X$ then $a_{\gamma'\gamma\gamma_X}$ and $a_{\gamma}$ differ by the diffeomorphism $c_{\gamma_X^{-1}\gamma^{-1}U}\to c_{\gamma^{-1}U}$ given by $t\mapsto \gamma_X t$. Hence, up to diffeomorphisms of the domain, the map $a_{\gamma}$ only depends on the class of $\gamma$ in $\Gamma_U\backslash \Gamma / \Gamma_X$. 

\begin{lemma} \label{lem:ll2}
    Let $\gamma\in \Gamma$. Then
    \begin{enumerate}[label=(\alph*),font=\upshape]
    \item if $\gamma^{-1}U\neq \Span X$ then the map $a_{\gamma}$ is a diffeomorphism onto a connected component of $p^\ast C(U)$;
    \item if $\gamma^{-1}U= \Span X$ then the map $C(X)  \to p^\ast C(U)$ induced by $a_{\gamma}$ is a diffeomorphism onto a connected component of $p^\ast C(U)$;
    \item the map 
    \begin{equation}\label{cce}
        \Gamma_U\backslash \Gamma / \Gamma_X\to \pi_0(p^\ast C(U)),
    \end{equation}
    sending $\gamma$ to the image of $a_{\gamma}$, is a bijection.
\end{enumerate}
\end{lemma}
\begin{proof}
First, we claim that in (a) we have $\Gamma_{\gamma^{-1}U}\cap \Gamma_X=1$. Indeed any element $\gamma\in \Gamma_{\gamma^{-1}U}\cap \Gamma_X$ would fix the projection of $\Span X$ to $U^\perp$, but there is only one non-trivial isometriy of the hyperbolic plane $(\gamma^{-1}U)^\perp$ which has a fixed vector and it is an involution, so as $\Gamma$ is torsion-free $\gamma=1$.

With that in mind, (a), (b), and (c) will all follow from the fact that the immersion
\[
    \bigsqcup_{\gamma\in \Sigma} (\Gamma_{\gamma^{-1}U}\cap \Gamma_X)\backslash c_{\gamma^{ - 1}U}\to p^\ast C(U)
\]
is a diffeomorphism, where $\Sigma\subset \Gamma$ is a set of representatives for $\Gamma_U\backslash \Gamma/ \Gamma_X$. 

The map in the other direction is given as follows: An element of $p^\ast C(U)$ is determined by a pair $(t,z)\in C(U)\times \Gamma_X\backslash \mathcal{H}$ with the same image in $M$. Letting $\tilde t\in c_U$ and $\tilde z\in \mathcal{H}$ be lifts of $t$ and $z$, this means that $\gamma \tilde z=\tilde t$ for a unique $\gamma\in \Gamma$. Writing $\gamma'\gamma_0\gamma_X=\gamma$, where $\gamma'\in \Gamma_U$, $\gamma_X\in\Gamma_X$ and $\gamma_0\in \Sigma$, it follows that $\gamma_0\gamma_X \tilde z = (\gamma')^{-1} \tilde t\in c_U$ and it is then easy to check that the inverse map is given by
\[
    (t,z)\mapsto \gamma_X\tilde z \in c_{\gamma_0^{ - 1}U}. 
\]
\end{proof}

\begin{remark}\label{rem:gbi}
Assume that $c(U)$ is closed. If $c(U)\neq c(X)$, then Lemma \ref{lem:ll2}(a) implies that every component of $p^\ast C(U)$ is noncompact. If $c(U)=c(X)$, then Lemma \ref{lem:ll2} implies that $p^\ast C(U)$ has exactly one closed component, corresponding to the identity coset in $\Gamma_X\backslash \Gamma / \Gamma_X$, and by construction the closed component maps diffeomorphically onto $\tilde c_X$ by $p^\ast c(U)$.
\end{remark}

%--------------------------------------------------------

\subsubsection{Proof of Proposition \ref{prp:ihi} and \ref{prp:bmt}}

By Lemma \ref{lem:ufl2}, we are left with proving that
\begin{equation}\label{ufi}
    \lim_{t\to\infty}\int_{p^\ast S} \tilde s^\ast\varphi^0(Xt^{1 /2}) = \iota(X;s),
\end{equation}
when $(t)$ holds.

Let $\tilde x = (x,t)$ be an intersection point of $s$ and $c(X)$. We recall that equivalently this just means that $\tilde x\in \tilde s^{-1}(\tilde c_X)$, that is $u(X)$ vanishes on $\tilde s(\tilde x)$.

If $\tilde x$ belongs to the interior of $p^\ast S$, then $(t^0)$ implies that $u(X)\circ \tilde s$ is regular at $\tilde x$ hence there exists an open neighborhood $\tilde x\in A_{\tilde x}\subset p^\ast S$ which is mapped diffeomorphically onto an open disk in $0\in B_{\tilde x}\subset\mathbb{R}^2$. Let $\varepsilon_{\tilde x}$ be equal to $1$ (resp. $-1$) if $u(X)\circ \tilde s$ is orientation preserving (resp. reversing). Then
\begin{equation}\label{ipi}
\lim_{t\to \infty}\int_{A_{\tilde x}} \tilde s^\ast \varphi^0(Xt^{1 /2}) =  \varepsilon_{\tilde x}\lim_{t\to\infty}\int_{B_{\tilde x}} te^{ - \pi t(u_1^2 + u_2^2)} du_1\wedge du_2 
= \varepsilon_{\tilde x}.
\end{equation}
If $\tilde x\in p^\ast C$ and it is a good boundary intersection point, then $(t_i)$ implies that $u(X)\circ \tilde s$ is regular at $\tilde x$. Let $\tilde c\subset p^\ast C$ be the connected component containing $\tilde x$. Since $\tilde x$ is good Remark \ref{rem:gbi} implies that $\tilde c$ is non-compact and by Lemma \ref{lem:ll1} and \ref{lem:ll2}(a) we see that $u(X)\circ \tilde s$ maps $\tilde c$ diffeomorphically onto a line in $\mathbb{R}^2$. Consider a neighborhood $\tilde x\in N\subset p^\ast S$ such that $s|_N$ admits an extension to a surface $N'\supset N$ and with the further property that $N- \tilde c$ and $N'- N$ are distinct connected components of $N'-\tilde c$. As before we can find a neighborhood $A_{\tilde x}'\subset N'$ which is mapped diffeomorphically onto an open disk $B_{\tilde x}\subset \mathbb{R}^2$ by $u(X)\circ \tilde s$. If $A_{\tilde x}:= A_{\tilde x}'\cap N$ then $A_{\tilde x}- \tilde c$ and $A_{\tilde x}'-A_{\tilde x}$ must map to each of the two connected components of $B_{\tilde x}-\ell$, that is the two half-disks seperated by $\ell$. Then
\begin{equation}\label{gbi}
\lim_{t\to \infty}\int_{A_{\tilde x}} \tilde s^\ast \varphi^0(Xt^{1 /2}) =  \varepsilon_{\tilde x}\lim_{t\to\infty}\int_{B_{\tilde x}\cap H} te^{ - \pi t(u_1^2 + u_2^2)} du_1\wedge du_2 = \frac{1}{2}\varepsilon_{\tilde x},
\end{equation}
where $H\subset \mathbb{R}^2$ denotes the half-plane containing the image of $A_{\tilde x}$.

Only the bad boundary intersection points remain now. Recall that these can only occur if $c(X)$ is closed, in which case they are precisely the points on the closed component $\tilde c\subset p^\ast C_i$ which maps diffeomorphically onto $\tilde c_X$. We use the following lemma in differential geometry, which is essentially just an extension of the inverse function theorem to a setting where the manifolds have boundary.

\begin{lemma}\label{lem:tdg}
    Let $X$ be a closed manifold and $u:X\times \mathbb{R}_{\geq0}\to \mathbb{R}^2$ a smooth map with $u(x,0)=0 $ and $\frac{\partial u}{\partial t} (x,0)\neq0$ for all $x\in X$ ($t$ denotes the coordinate on $\mathbb{R}_{\geq0}$). Let $r:X\times \mathbb{R}_{\geq0}\to \mathbb{R}_{\geq 0}$ be defined by $r(x,t)=\sqrt{u_1(x,t)^2+u_2(x,t)^2}$. Then the map
    \[
        \phi:X\times \mathbb{R}_{\geq0}\to X\times \mathbb{R}_{\geq_0} \hspace{1cm} (x,t)\mapsto (x,r(x,t)).
    \]
    restricts to a continuously differentiable bijection $A\to B$ between open neighborhoods $A$ and $B$ of $X\times \{0\}$ for which the inverse $B\to A$ is also continuously differentiable.
\end{lemma}

We apply the lemma to a collar neighborhood of $\tilde c$ in $p^\ast S$ and the map $u=u(X)\circ \tilde s$. We obtain open neighborhoods $A_i\subset p^\ast S$ and $B_i\subset \tilde c\times \mathbb{R}_{\geq0}$ of $\tilde c$ and a continuously differentiable bijection $\phi:A_i\to B_i$ with continuously differentiable inverse, such that the second coordinate of $\phi$ is given by $r(X)\circ\tilde s$. By possibly further shrinking $B_i$, we may assume that $B_i=\tilde c\times [0,\varepsilon)$ for some $\varepsilon>0$.

We then use the following geometric observation: Suppose that $C$ is a closed oriented curve and that $a:C\times [0,\varepsilon)\to \mathbb{R} / \mathbb{Z}$ is a continuous function which is smooth on $C\times (0,\varepsilon)$. For $r\in [0,\varepsilon)$ let $a_r$ denote the map $t\mapsto a(t,r)$. We then have $\int_{C\times \{r\}} da = \deg(a_r)= \deg (a_0)$ for every $r\in (0,\varepsilon)$ by homotopy. Also
\[
    2\pi\lim_{t\to \infty} \int_0^\varepsilon tr e^{ - \pi t r^2} dr = 1
\]
hence

\begin{equation}\label{ilc}
    2\pi \lim_{t\to \infty} \int_{C\times [0,\varepsilon)} t r e^{ - \pi tr^2} dr\wedge da = \deg(a_0).
\end{equation}
We want to apply this with $\tilde c$. To define the map $a$, let
\[
    a(X) :=\, \begin{cases} \frac{1}{2\pi}\arctan\left(\frac{u(X_2)}{u(X_1)}\right) & u(X_1) > 0 \\ 
        \frac{1}{2\pi}\arctan\left(\frac{u(X_2)}{u(X_1)}\right) + \frac{1}{2}& u(X_1) < 0 \\ 
        \frac{1}{4}\sgn(u(X_2)) & u(X_1) = 0
    \end{cases}.
\]
This defines a smooth function $a(X):(\mathcal{H}-c_X)\to\mathbb{R} / \mathbb{Z}$, and hence $a:= a(X)\circ \tilde s \circ \phi^{-1}$ is a smooth function on $\tilde c\times (0,\varepsilon)$. We must show that $a(t,r)$ converges uniformly as $r\to 0$. Let $u:= u(X)\circ \tilde s \circ \phi^{-1}$ hence $u(t,r) = r(\cos(2\pi a(t,r)),\sin(2\pi a(t,r)))$. Let $a_0$ be defined in the same way as $a(X)$ but with $\frac{\partial u}{\partial r}(-,0)$ instead of $u(X)$. By Taylor's theorem for $2$ variables then $\frac{u(t,r)}{r}$ converges uniformly to $\frac{\partial u}{\partial r}(t,0)$ hence we see that $a(t,r)$ converges uniformly to $a_0(t)$.

Finally, we show that $\deg(a_0)$ equals the winding number in the definition of $\iota_i(X;s)$. Let $R$ be a section of $N_{C_i / S}$ with no zeroes, and let $\widetilde R$ be the pullback of $R$ to a section of $N_{\tilde c / p^\ast S}$. The two sections $\widetilde R$ and $\phi_\ast^{-1} (\frac{\partial}{\partial r})$ are both positively oriented bases of $N_{\tilde c / p^\ast S}$, hence they are homotopic. In particular
\[
    du(X)(d\tilde s(\widetilde R)) = du\left(\frac{\partial}{\partial r}\right) = \frac{\partial u}{\partial r}( - ,0).
\]
Note that the left-hand side is homotopic to the map $f_i$ from \eqref{wnm} and the right-hand side is homotopic to $a_0$. Hence by \eqref{ilc} we obtain
\begin{equation}\label{bbi}
    \lim_{t\to\infty}\int_{A_i} \tilde s^\ast \varphi^0(Xt^{1 /2}) = \deg f_i.
\end{equation}
We may pick $A_{\tilde x}$ and $A_i$ small enough so that they do not intersect each other. Then \eqref{ufi} follows by \eqref{ipi}, \eqref{gbi}, \eqref{bbi}, and the second part of Lemma \ref{lem:ufl2} applied to 
$A:= p^\ast S - \bigcup_{\text{good}\;\tilde x} A_{\tilde x} - \bigcup_{i\in I_2} A_i$.

%--------------------------------------------------------

\subsection{An integral formula}\label{sec:scd}

In this subsection we let $(V,Q)$ be a real quadratic space with signature $(1,1)$. We fix an infinite cyclic subgroup $\Gamma$ in the connected component of the identity in $\SO(V)$, an orientation of $V$, and a connected component $\mathcal{D}^+\subset\mathcal{D}$ of the Grassmannian of oriented negative lines in $V$.

Let $X\in V^2$. For $z\in \mathcal{D}$ we let $X(z)\in V$ denote the positively oriented basis vector of the line corresponding to $z$ with $Q(X(z))=1$. We then define $u(X):\mathcal{D}^+\to\mathbb{R}^2$ and $\psi^0(X)\in\Omega^1(\mathcal{D}^+)$ by the same formula as in Section \ref{sec:fai}. If $X$ is not a basis of $V$ then $\psi^0(X)=0$. If $X$ is a basis of $V$ we let 
\begin{equation}\label{ebd}
    \varepsilon(X) = \begin{cases} + 1 & X\;\text{is positively oriented} \\  -1 & X\;\text{is negatively oriented}\end{cases} .
\end{equation}
If $T$ is a $2\times2$ matrix with real eigenvalues $\lambda_1\geq\lambda_2$, we let $\Delta(T)=\lambda_1-\lambda_2$. Equivalently, we have
\begin{equation}\label{dtf}
    \Delta(T) = ((\tr T)^2 - 4\det T)^{1 /2}.
\end{equation}
If $X\in V^2$ is a basis, then $T=Q(X)$ is an indefinite symmetric matrix, hence its eigenvalues are real and $\lambda_1>0>\lambda_2$. In particular $\det(Q(X))<0$ and 
\begin{equation}\label{pec}
    \Delta(Q(X))>|\tr Q(X)|.
\end{equation}

Let $M=\Gamma\backslash \mathcal{D}^+$. 
We assume that $X$ satisfies the obvious analogue of Assumption \ref{ass:agx} so that the averaged form $\alpha(X,\psi^0)$ defines a smooth form on $M$.

We orient $\mathcal{D}$ and hence $M$ as follows. We identify the tangent space at $z\in \mathcal{D}$ with the orthogonal complement of $X(z)$ in $V$. Then we declare a tangent vector $Y$ to be positively oriented if $Y,X(z)$ is a positively oriented basis of $V$.
\begin{proposition}\label{prp:iyf}
We have
\[
    \int_M \alpha(X ,\psi^0) = \varepsilon(X)|\det Q(X)|^{1 /2 }  e^{ 2\pi \tr Q(X)} K_0(2\pi  \Delta(Q(X))),
\]
where $K_0(x) = \int_0^\infty e^{-x \cosh t}dt$ is a modified Bessel function of the second kind. 
\end{proposition}
\begin{proof}
By unfolding
\begin{equation}\label{duf}
    \int_M \alpha(X ,\psi^0) = \int_{\mathcal{D}^+} \psi^0(X ).
\end{equation}
We fix a positively oriented basis $e_1,e_2\in V$ such that $Q(e_1)=Q(e_2)=0$ and $(e_1,e_2)=1$. There is then an orientation preserving identification $\mathcal{D}=\mathbb{R} - \{0\}$ such that we have $X(z)=-z^{-1}e_1+ze_2$ for $z\neq0$. By potentially using the basis $-e_1,-e_2$ instead of $e_1,e_2$ we can ensure that $\mathcal{D}^+=\mathbb{R}_{>0}$. 

Let $X_i=a_ie_1+b_ie_2$. We have $u(X_i) = a_i z-b_i z^{-1}$ and $du(X_i) = (a_iz+b_i z^{-1})\frac{dz}{z}$, hence a straightforward calculation gives
\[
    \psi_0(X) = \frac{1}{2}(a_1b_2 - b_1a_2) \frac{dz}{z}
\]
and
\[
    r(X)^2 = (a_1^2 + a_2^2)z^{ - 2} - 2(a_1b_1 + a_2b_2) + (b_1^2 + b_{2}^2) z^2.
\]
It follows that
\begin{equation}\label{cif}
    \int_{\mathcal{D}^+} \psi^0(X) = \frac{1}{2}(b_1a_2 - a_1b_2)e^{ 2\pi(a_1b_1 + a_2b_2)} \int_{0}^\infty e^{ -\pi((a_1^2 + a_2^2)z^{ - 2} + (b_1^2 + b_{2}^2) z^2)} \frac{dz}{z}.
\end{equation}
If $\alpha,\beta>0$ then the substitution $z = (\frac{\alpha}{\beta}e^t)^{1 /2}$ yields
\[
    \int_{0}^\infty e^{ - \pi(\alpha^2 z^{ - 2} + \beta^2 z^2)} \frac{dz}{z} = \frac{1}{2}\int_{ - \infty}^\infty e^{ -2\pi\alpha \beta\cosh t} dt = K_0(2\pi \alpha \beta).
\]
Taking $\alpha=(a_1^2 + a_2^2)^{1 /2}$ and $\beta = (b_1^2+b_2^2)^{1 /2}$, then in conjunction with \eqref{duf} and \eqref{cif} this gives an explicit formula for $\int_M\alpha(X,\psi^0)$. It remains to verify that the constants are correct. A simple calculation gives
\[
    \tr Q(X) = a_1b_1 + a_2b_2\hspace{1cm}\det Q(X) = - \frac{1}{4}(a_1b_2 - b_1a_2)^2
\]
hence by \eqref{dtf}
\[
    \Delta(Q(X)) = [(a_1^2 + a_2^2)(b_1^2 + b_2^2)]^{1 /2} = \alpha \beta,
\]
and lastly, we have
\[
    \varepsilon(X) = \sgn\det \begin{pmatrix} a_1 & a_2 \\ b_1 & b_2 \end{pmatrix} = \sgn(a_1b_2 - b_1a_2).
\]
\end{proof}

\begin{corollary}\label{cor:dcc}
    The form $\int_M \alpha(X t^{1 /2},\psi^0)\frac{dt}{t}$ is integrable on $[1,\infty)$ and
    \[
        \int_1^\infty \int_M \alpha(X t^{1 /2},\psi^0)\frac{dt}{t} = \varepsilon(X) W^\ast(\tr(Q(X)), 4\abs{\det(Q(X))}).
    \]
    where
    \[
        W^\ast(x_1,x_2) := \frac{1}{2}\sqrt{x_2} \int_1^\infty e^{2\pi t x_1} K_0\left(2\pi t \sqrt{x_1^2 + x_2}\right)dt \hspace{1cm} x_1,x_2\in \mathbb{R}, x_2 > 0.
\]
\end{corollary}
\begin{proof}
We use the inequality
\begin{equation}\label{mbfi}
    K_0(x) < \sqrt{\frac{\pi}{2x}} e^{-x},\hspace{1cm} x>0,
\end{equation}
which follows from $\cosh(t)\geq 1+\frac{t^2}{2}$. It follows that
\begin{equation}\label{bfi}
    e^{2\pi t\tr Q(X)}K_0(2\pi t\Delta(Q(X))) < \frac{1}{2 \sqrt{t\Delta(Q(X))}}e^{ - 2\pi t (\Delta(Q(X)) - \tr Q(X))}.
\end{equation}
Then \eqref{pec} and Proposition \ref{prp:iyf} imply the integrability and the formula.
\end{proof}

%--------------------------------------------------------

\section{Modular generating series}\label{sec:mgs}

In this section, we state and prove the main results of the paper.

%--------------------------------------------------------

\subsection{Special cycles}

Let $(V,Q)$ be a nondegenerate rational quadratic space. We assume that $V_\mathbb{R}$ has signature $(3,1)$, and we also fix an orientation of $V(\mathbb{R})$. Let $\mathcal{H}$ be a connected component of the Grassmannian of oriented negative lines in $V(\mathbb{R})$. Also, let $G=\SO(V)$ and let $\Gamma\subset G(\mathbb{Q})$ be a torsion-free arithmetic subgroup that is contained in the connected component of $G(\mathbb{R})$. 

This data determines the arithmetic hyperbolic $3$-manifold $M := \Gamma\backslash \mathcal{H}$. Using the orientation of $V(\mathbb{R})$ we orient $\mathcal{H}$ and $M$ as described in Subsection \ref{sec:tst}. If $U\subset V$ is a rational subspace such that $U_{\mathbb{R}}$ is positive definite and oriented, then $c(U):=c(U_{\mathbb{R}})$ defines a special cycle in $M$. If $\dim U=1$, then $c(U)$ is an immersed arithmetic hyperbolic surface. If $\dim U=2$, then $c(U)$ is either a hyperbolic geodesic, or an infinite length geodesic connecting two of the cusps of $M$. The first case occurs precisely when $U^\perp$ is anisotropic. Let $d(V)\in \mathbb{Q}^\times / (\mathbb{Q}^\times)^2$ be the discriminant of $V$, hence by the signature condition $d(V)<0$. If $X\in V(\mathbb{Q})^2$ and $Q(X)\in\Sym_2(\mathbb{Q})$ is positive definite, then it follows that $c(X)$ defines a hyperbolic geodesic when $-d(V)\det Q(X)$ is not a rational square. By Remark \ref{rem:asa}, this is equivalent to $\sqrt{-d(V)}\tan\theta\not\in\mathbb{Q}^\times$, where $\theta$ is the angle between the arithmetic hyperbolic surfaces defined by $X_1$ and $X_2$ along $c(X)$.

\begin{example}\label{exa:sce}
Let
\[
    V = \left\{\begin{pmatrix}
        -x & y \\ z & \overline x
    \end{pmatrix} : x\in \mathbb{Q}(i),\;y,z\in \mathbb{Q} \right\}.
\]
with $Q(X)=-\det X$. We then have an isomorphism
\begin{equation}\label{gqi}
    G(\mathbb{Q}) \cong \{g\in \GL_2(\mathbb{Q}(i)):\det g\in \mathbb{Q}^\times\} / \mathbb{Q}^\times,
\end{equation}
with the action of the right-hand side on $V$ given by \eqref{gaf} \cite[10.2.2]{MR1937957}. In particular, there is a short exact sequence
\[
    1\to \PSL(\mathbb{Q}(i))\to G(\mathbb{Q})\to \mathbb{Q}^\times / (\mathbb{Q}^\times)^2\to 1,
\]
where the second map is the spinor norm, given in terms of \eqref{gqi} by $g\mapsto\det g$. The subgroup $\PSL_2(\mathbb{Z}[i])\subset G(\mathbb{Q})$ is arithmetic, hence our setup applies to any torsion-free finite index subgroup of $\PSL_2(\mathbb{Z}[i])$. We have $d(V)=-1$, so the geodesics $c(X)$ are closed when $\tan\theta\not\in\mathbb{Q}$.
\end{example}

\subsection{Main theorems} \label{sec:mts}

Let $\varphi_f\in \mathscr{S}(V(\mathbb{A}_f)^2)$ be a finite Schwartz-Bruhat function which is fixed by $\Gamma$ and let $T\in \Sym_2(\mathbb{Q})$. For convenience, we also fix a lattice $\Lambda\subset V(\mathbb{Q})^2$ that contains all $X\in V(\mathbb{Q})^2$ such that $\varphi_f(X)\neq0$. 

Let $s$ be a Seifert surface in $M$ whose boundary is a union of hyperbolic geodesics. We define 
\[
    \iota(T,\varphi_f;s) := \sum_{\substack{X\in V(\mathbb{Q})^2 / \Gamma \\ Q(X) = T}} \varphi_f(X) \iota(X;s).
\]
By Proposition \ref{prp:bmt} this is zero when $T$ is not positive definite. When $T$ is positive definite then the sum contains finitely many non-zero terms by \cite[11.9]{borel2019introduction}, see also \cite[p.132]{MR1079646}.

For $v\in \Sym_2(\mathbb{R})$ positive definite let
\begin{equation}\label{qav}
    \Theta(T,\varphi_f,\psi^0,v) = \sum_{\substack{X\in V(\mathbb{Q})^2 \\ Q(X) = T }} \varphi_f(X) \psi^0(Xa).
\end{equation}
where $a\in \GL^+_2(\mathbb{R})$ is any matrix such that $v = a\tp{a}$. Note that by Lemma \ref{lem:bpl}(b) this does not depend on the choice of $a$, and by Lemma \ref{lem:nlc} applied to the lattice $\Lambda a$ it converges to a smooth $\Gamma$-invariant form on $\mathcal{H}$ and hence we may identify it with a form on $M$. 

Suppose that $c$ is a hyperbolic geodesic in $M$. We further define 
\begin{equation}\label{bct}
    \beta(T,\varphi_f,v;c) := \int_1^\infty \int_c \Theta(T,\varphi_f,\psi^0,vt) \frac{dt}{t}.
\end{equation}
and for $\tau=u+iv\in \mathfrak{H}$
\begin{align}\label{igs}
    I(\tau,\varphi_f;s) &:= \sum_{T\in \Sym_2(\mathbb{Q})} \iota(T,\varphi_f;s) q^T, \\ \label{ggs}
    G^\ast(\tau,\varphi_f;c) &:= \sum_{T\in \Sym_2(\mathbb{Q})} \beta(T,\varphi_f,v;c) q^T .
\end{align}

\begin{theorem}\label{thm:mta} With notation as above
    \begin{enumerate}[label=(\roman*)]
        \item the series in \eqref{igs} converges absolutely on $\mathfrak{H}$;
        \item for all $T\in \Sym_2(\mathbb{Q})$ the outer integral in \eqref{bct} converges, and the series in \eqref{ggs} converges absolutely on $\mathfrak{H}$.
    \end{enumerate}
\end{theorem}
We wish to also formulate a modularity result which relates the two series. Let $L\subset V(\mathbb{Q})$ be an integral lattice and let $L^\#\subset V(\mathbb{Q})$ be its dual lattice. Let $D_L = L^\# / L$ denote the discriminant group of $L$. We assume that $\Gamma$ acts trivially on $D_L$. For $\mu\in (D_L)^2$ let $\varphi_f(\mu)\in\mathscr{S}(V(\mathbb{A}_f)^2)$ be the indicator function of 
\[
    \{X\in (L^\#)^2\otimes \widehat{\mathbb{Z}} : X\equiv \mu \mod L^2\otimes \widehat{\mathbb{Z}}\} \subset V(\mathbb{A}_f)^2
\]
Let $(\mathfrak{e}_{\mu})_{\mu\in (D_L)^2}$ be the standard basis of $\mathbb{C}[(D_L)^2]$. We write elements $a\in\mathbb{C}[D_L^2]$ as $a=\sum_{\mu\in D_L^2}a_\mu \mathfrak{e}_\mu$. We then define $\iota_L(T;s),\beta_L(T,v;c)\in\mathbb{C}[D_L^2]$ by
\begin{align*}
    \iota_L(T;s)_\mu &:= \iota(T,\varphi_f(\mu);s),
    \\
    \beta_L(T,v;c)_\mu &:= \beta(T,\varphi_f(\mu),v;c),
\end{align*}
and similarly we define the functions $I_L(-;s),G_L(-;c):\mathfrak{H}\to \mathbb{C}[(D_L)^2]$ by
\begin{align*}
    I_L(\tau;s)_\mu &:= I(\tau,\varphi_f(\mu);s),
    \\
    G^\ast_L(\tau,v;c)_\mu &:= G^\ast(\tau,\varphi_f(\mu),v;c).
\end{align*}
We let $\rho_L:\Sp_4(\mathbb{Z})\to \Aut \mathbb{C}[(D_L)^2]$ denote the finite Weil representation, which will be given a more detailed description in Subsection \ref{sec:mwp}. We then say that a function $f:\mathfrak{H}\to \mathbb{C}[(D_L)^2]$ is modular of weight $k$ and level $1$ if
\[
    f(\gamma \tau) = \det(c\tau +d)^k \rho_L(\gamma) f(\tau) \hspace{1cm} \gamma = \begin{pmatrix}
        a & b \\ c & d
    \end{pmatrix}\in \Sp_4(\mathbb{Z}),\;\tau\in \mathfrak{H}.
\]
\begin{theorem} \label{thm:mtb}
    Assume that $c$ is the boundary of $s$. Then the function 
    \[
        I_L(\tau;s) + G^\ast_L(\tau;c)
    \]
    is modular of weight $2$ and level $1$. 
\end{theorem}
\begin{remark}\label{rem:LlN}
    Let $N$ be the level of $L$. Since $\rho_L$ is trivial on the congruence subgroup $\Gamma(N)\subset \Sp_4(\mathbb{Z})$, it follows that $I_L(\tau;s)_\mu+G^\ast_L(\tau;c)_\mu$ is modular of weight $2$ and level $N$, for every $\mu\in D_L^2$.
\end{remark}

Lastly, we obtain the following polynomial estimate for the coefficients of $I_L(\tau;s)$.

\begin{theorem} \label{thm:mtc}
    For every $\mu\in D_L^2$ and every $\varepsilon>0$
    \[
        |\iota_L(T;s)_\mu| \ll \det(T)^{3/2 + \varepsilon}.
    \]
\end{theorem}

\begin{example}\label{exa:mte}
Continuing Example \ref{exa:sce}, we have an integral lattice
\[
    L = \left\{\begin{pmatrix}
        -x & y \\ z & \overline x
    \end{pmatrix} : x\in \mathbb{Z}[i],\;y,z\in \mathbb{Z} \right\}.
\]
The dual lattice is given by
\[
    L^\# = \left\{\begin{pmatrix}
        -b /2 & c \\ -a & \overline b /2
    \end{pmatrix} : b\in \mathbb{Z}[i],\;a,c\in \mathbb{Z} \right\},
\]
and the map $X\mapsto b$ gives an isomorphism $D_L \cong \mathbb{Z}[i] /(2)$. A simple calculation shows that $\PSL_2(\mathbb{Z}[i])\subset G(\mathbb{Q})$ stabilizes $L$ and acts trivially on $D_L$. %Many finite index subgroups $\Gamma\subset \PSL_2(\mathbb{Z}[i])$ are neat, for example we explore one such group in the next section. 

Theorem \ref{thm:imt} and Theorem \ref{thm:ist} follow from Remark \ref{rem:LlN} and Theorem \ref{thm:mtc} for $\mu=0$, applied to this example. 
\end{example}

%--------------------------------------------------------

\subsection{Convergence in signature $(1,1)$}

Let $(V,Q)$ be a rational quadratic space such that $V_{\mathbb{R}}$ has signature $(1,1)$. In this subsection, we state and prove a version of Theorem \ref{thm:mta}(ii) in this setup. 

Assume that $V(\mathbb{Q})$ is anisotropic. Let $G=\SO(V)$ and let $\Gamma\subset G(\mathbb{Q})$
be an infinite cyclic subgroup countained in the connected component of the identity in $G(\mathbb{R})$. As in Section \ref{sec:scd} we fix an orientation of $V(\mathbb{R})$ and a connected component $\mathcal{D}^+$ of the Grassmannian of negative oriented lines in $V(\mathbb{R})$. Thus $M=\Gamma \backslash \mathcal{D}^+$ is an oriented closed connected curve.

Let $\varphi_f\in \mathscr{S}(V(\mathbb{A}_f)^2)$ be fixed by $\Gamma$ and let $T\in \Sym_2(\mathbb{Q})$. As in \eqref{qav} we let
\[
    \Theta(T,\varphi_f,\psi^0,v) = \sum_{\substack{X\in V(\mathbb{Q})^2 \\ Q(X) = T }} \varphi_f(X) \psi^0(Xa).
\]
where $v\in \Sym_2(\mathbb{R})$ is positive definite and $v=a\tp{a}$. We then define
\begin{equation}\label{bcd}
    \beta(T,\varphi_f,v) :=\int_1^\infty\int_M \Theta(T,\varphi_f,\psi^0,vt)\frac{dt}{t} .
\end{equation}

\begin{proposition}\label{prp:gsp}
    Let $\tau=u+iv\in \mathfrak{H}$. For all $T\in \Sym_2(\mathbb{Q})$ the outer integral in \eqref{bcd} converges and the series
    \[
        \sum_{T\in \Sym_2(\mathbb{Q})}\beta(T,\varphi_f,v) q^T
    \]
    converges absolutely.
\end{proposition}

To prove Proposition \ref{prp:gsp}, we need an arithmetic input. This takes the form of bounds on certain "indefinite" representation numbers, which we define shortly. First we prove a preliminary result about real quadratic number fields.

\begin{lemma}\label{lem:brt}
    Let $F$ be a real quadratic number field, and let $\mathcal{O}\subset F$ be an order in $F$. Let $x\mapsto x'$ be the non-trivial automorphism of $F$ and let $\Nm(x)=x x'$ and $\Tr(x)=x+x'$ be the norm and trace maps, respectively. Let $I\subset F$ be a fractional ideal of $\mathcal{O}$. Then
    \begin{enumerate}[label=(\roman*)]
        \item for every $\eta\in F$ and every $\varepsilon>0$
        \[
            \# \frac{\{\alpha\in I +\eta: \Nm(\alpha) = t \}}{\{ x\in \mathcal{O}^\times : \Nm(x) = 1 , x(I + \eta)\subset I + \eta \}} \ll |t|^\varepsilon, \hspace{1cm} t\in \mathbb{Q}.
        \]
        \item for every $\eta=(\eta_1,\eta_2)\in F^2$ and every $\varepsilon>0$
        \begin{align*}
            \#\frac{\{\alpha =(\alpha_1,\alpha_2)\in I^2 +\eta: \Nm(\alpha_i) = t_i \;\text{for}\; i=1,2\;\text{and}\; \Tr(\alpha_1\alpha_2') = t_0 \}}{\{x\in \mathcal{O}^\times : \Nm(x)=1 : x(I^2+\eta)\subset I^2+\eta \}} \\ \ll |t_0^2-4t_1t_2|^\varepsilon, \hspace{1cm} t_0,t_1,t_2\in \mathbb{Q}.
        \end{align*}
    \end{enumerate}
\end{lemma}
Here both quotients are orbit spaces.
\begin{proof}
(i) For some integer $m$ we have $mI\subset \mathcal{O}_F$ and $m\eta\in \mathcal{O}_F$. Then multiplication by $m$ provides an injection
\[
    \{\alpha\in I +\eta: \Nm(\alpha) = t \}\hookrightarrow \{\alpha\in \mathcal{O}_F : \Nm(\alpha) =m^2 t\}.
\]
As the group $\{ x\in \mathcal{O}^\times : \Nm(x) = 1 , x(I + \eta)\subset I + \eta \}$ is of finite index in $\mathcal{O}_F^\times$, this shows that 
\[
    \# \frac{\{\alpha\in I +\eta: \Nm(\alpha) = t \}}{\{ x\in \mathcal{O}^\times : \Nm(x) = 1 , x(I + \eta)\subset I + \eta \}} \ll \#\{\text{ideals in $\mathcal{O}_F$ of norm $m^2t$}\}.
\]
The number of ideals in $\mathcal{O}_F$ of norm $x$ is $O( x^\varepsilon)$, hence we conclude.

(ii) First we prove that the size of the orbit space in question is  $O(|t_2|)^\varepsilon$. We claim that for fixed $t_0,t_1,t_2$ then the map
\begin{align*}
    \frac{\{\alpha =(\alpha_1,\alpha_2)\in I^2 +\mu: \Nm(\alpha_i) = t_i \;\text{for}\; i=1,2\; \text{and} \Tr(\alpha_1\alpha_2') = t_0\}}{\{ x\in \mathcal{O}^\times : \Nm(x) = 1 , x(I^2 + \eta)\subset I^2 + \eta \}} \\
    \to \frac{\{\alpha\in I +\eta_1: \Nm(\alpha) = t_1 \}}{\{ x\in \mathcal{O}^\times : \Nm(x) = 1 , x(I^2 + \eta)\subset I^2 + \eta \}},
\end{align*}
given by $\alpha\mapsto \alpha_1$, has fibers with at most two elements. Or equivalently, given $\alpha_1\in I+\eta_1$ with $\Nm(\alpha_1)=t_1$ there exists at most two elements $\alpha_2\in I+\eta_2$ such that
\[
    \alpha_2\alpha_2' = t_2\qquad \text{and}\qquad\alpha_1\alpha_2' + \alpha_2\alpha_1' = t_0
\]
Indeed, the two equations combine to show that $\alpha_2$ is a solution of the the quadratic equation $\alpha_1' \alpha_2^2 -t_0\alpha_2 +\alpha_1 t_2 = 0$.

The claimed estimate now follows since the group $\{ x\in \mathcal{O}^\times : \Nm(x) = 1 , x(I^2 + \eta)\subset I^2 + \eta \}$ is of finite index in the group $\{x\in \mathcal{O}^\times : \Nm(x) = 1 , x(I^2 + \eta_1)\subset I^2 + \eta_1\}$.

To obtain the bound by $|t_0^2-4t_1t_2|^\varepsilon$, observe that the conditions $\Nm(\alpha_i)=t_i$ and $\Tr(\alpha_1\alpha_2')=t_0$ are equivalent to $\Nm(\alpha_1 x+\alpha_2y) = t(x,y)$, where $t$ is the (rational) binary quadratic form $t(x,y) = t_1x^2+t_0xy+t_2y^2$. In particular $t$ must be indefinite, otherwise the set in question will be empty. For $(\begin{smallmatrix}
    a & b \\ c & d
\end{smallmatrix})\in \SL_2(\mathbb{Z})$ the map 
\[
    \alpha\mapsto \tilde\alpha = (\alpha_1 a+\alpha_2 c, \alpha_1 b + \alpha_2 d)
\]
provides a bijection between $\alpha\in I^2+\eta$ satisfying $\Nm(\alpha_1x+\alpha_2y) = t(x,y)$ and $\beta\in I^2+\tilde\eta$ satisfying $\Nm(\beta_1x+\beta_2y)= t(ax+by,cx+dy)$. Also, the condition $x(I^2+\eta)\subset I^2+\eta$ is equivalent to $x(I^2+\tilde\eta)\subset I^2+\tilde\eta$. By the reduction theory of indefinite binary quadratic forms we may therefore assume that $|t_2| < \sqrt{t_0^2-4t_1t_2}$, which completes the proof.
\end{proof}

Recall the function $\varepsilon(X)$ defined by \eqref{ebd} for a basis $X\in V(\mathbb{R})^2$. We extend it to all of $V(\mathbb{R})^2$ by setting $\varepsilon(X)=0$, when $X$ is linearly dependent. We then let
\[
        \rho(T,\varphi_f) := \sum_{\substack{X\in V(\mathbb{Q})^2 / \Gamma \\ Q(X) = T}} \varphi_f(X)\varepsilon(X).
\]
Using Lemma \ref{lem:brt}, we can prove that this is finite and give an explicit bound.
\begin{lemma}\label{lem:rnb}
    For every $\varphi_f$ and every $\varepsilon>0$
    \[
        |\rho(T,\varphi_f)| \ll |\det T|^\varepsilon.
    \]
\end{lemma}

\begin{proof}
    It suffices to prove the lemma when $\varphi_f$ is the indicator function of $L^2+\eta$ for some integral lattice $L\subset V$ and $\eta=(\eta_1,\eta_2)\in V(\mathbb{Q})^2$. Furthermore we may assume that $\gcd\{Q(x):x\in L\}=1$. Then
    \[
        |\rho(T,\varphi_f)| \leq \sum_{\substack{X\in V(\mathbb{Q})^2 / \Gamma \\ Q(X) = T }} |\varphi_f(X)| = \# (\{\;X\in L^2 + \eta : Q(X) = T\;\} /\Gamma).
    \]
    Let $\mathcal{O}$ be the unique quadratic order of discriminant equal to the discriminant of $L$. By \cite[p. 82]{arakawa2014bernoulli} there exists an isometry between $(L,Q)$ and $(I,N(I)^{-1}\Nm)$, for some ideal $I\subset \mathcal{O}$. This identifies the isometry group of $(V,Q)$ with $F^\times$, and hence $\Gamma$ with a finite subgroup of $\{ x\in \mathcal{O}^\times : \Nm(x) = 1 , x(I + \eta)\subset I + \eta \}$. Thus the present lemma now follows from Lemma \ref{lem:brt}.
\end{proof}

\begin{proof}[Proof of Proposition \ref{prp:gsp}]

We make use of the short-hand notation $T^a = \tp{a} T a$, for $a\in \GL^+_2(\mathbb{R})$. Note that $Q(Xa) = Q(X)^a$. We fix $a$ such that $v= a\tp{a}$.

We have $\beta(T,\varphi_f,v)=0$ if $T$ is not indefinite, and if $T$ is indefinite then by  Corollary \ref{cor:dcc}
\begin{equation}\label{ebf}
    \beta(T,\varphi_f,v) = \rho(T,\varphi_f) W^\ast(\tr(Tv),4 |\det(Tv)|)
\end{equation}
We have $\{Q(X):\varphi_f(X)\neq 0\}\Sym_2(\mathbb{Z}) \subset \frac{1}{N}\Sym_2(\mathbb{Z})$ for some positive integer $N$. Since $|q^T| = e^{-2\pi \tr T^a}$, Lemma \ref{lem:rnb} leaves us to prove that
\begin{equation}\label{ueb}
    \sum_{\substack{T\in \frac{1}{N}\Sym_2(\mathbb{Z})^a \\ \text{indefinite}}} |\det T|^{\varepsilon}W^\ast(\tr(T),4 |\det(T)|) e^{ -2\pi \tr (T)}
\end{equation}
converges. 

We now use the following estimate which is easily obtained from \eqref{mbfi} 
\begin{equation}\label{Wfe}
    W^\ast(x_1,x_2) \ll \exp(- 2\pi \left(\sqrt{x_1^2 + x_2}-x_1\right)).
\end{equation}
This gives a bound by
\[
    \sum_{\substack{T\in \frac{1}{N}\Sym_2(\mathbb{Z})^a \\ \text{indefinite}}} \abs{\det T}^\varepsilon e^{ -2\pi \Delta(T)}.
\]
For indefinite $T=(\begin{smallmatrix} t_1 & t_0  \\ t_0  & t_2 \end{smallmatrix})\in \Sym_2(\mathbb{R})$ we have $|t_0|,|t_1|,|t_2|\leq \Delta(T)$. As $\frac{1}{N}\Sym_2(\mathbb{Z})^a$ is a lattice inside $\Sym_2(\mathbb{R})$, we infer that %$\Delta(T)$ is bounded from below and using \eqref{Wfe} this gives a bound of \eqref{ueb} by
$
    \# \{T\in \frac{1}{N}\Sym_2(\mathbb{Z})^a : \Delta(T) < R\} \ll R^3.
$
Thus we reduce to the fact that $\sum_{n=1}^\infty r^{n} <\infty$ for any $0<r<1$.
\end{proof}

%--------------------------------------------------------

\subsection{Proof of Theorem \ref{thm:mta}}\label{sec:tap}

First we show that Theorem \ref{thm:mta}(i) follows from Theorem \ref{thm:mta}(ii). Analogously to \eqref{qav} we let 
\[
    \Theta(T,\varphi_f,\varphi^0,v) = \sum_{\substack{X\in V(\mathbb{Q})^2 \\ Q(X) = T }} \varphi_f(X) \varphi^0(Xa).
\]
By Lemma \ref{lem:xai} we have
\[
    \lim_{t\to\infty}\int_s \Theta(T,\varphi_f,\varphi^0,vt) =  \iota(T,\varphi_f;s) 
\]
for any $v$. Let $c=\partial s$. By Lemma \ref{lem:bpl}(e) and Stokes theorem we have
\[
    \int_s \Theta(T,\varphi_f,\varphi^0,v) = \int_s \Theta(T,\varphi_f,\varphi^0,v t') + \int_c \int_1^{t'} \Theta(T,\varphi_f,\psi^0,v t) \frac{dt}{t}.
\]
We can switch the integrals over $c$ and $t$ by compactness. Hence, taking $t'\to\infty $ and assuming Theorem \ref{thm:mta}(ii) we have
\[
    I(\tau,\varphi_f;s) = \sum_{T\in \Sym_2(\mathbb{Q})}\int_s\Theta(T,\varphi_f,\varphi^0,v) q^T - G(\tau,\varphi_f;c)
\]
where the first series converges absolutely by lemma \ref{lem:nlc} applied to the lattice $\Lambda a$ and by the compactness of $S$.

Now we turn to the proof of Theorem \ref{thm:mta}(ii), and we begin with a particular but central special case. Fix a subspace $U\subset V$ with $U(\mathbb{R})$ oriented and of signature $(1,1)$ and let $V' = U^\perp$.

\begin{lemma}\label{lem:rdl}
    Suppose that $\varphi_f=\varphi_f'\otimes \varphi_f''$ for Schwartz functions $\varphi_f''\in \mathscr{S}(U(\mathbb{A}_f)^2)$ and $\varphi_f'\in \mathscr{S}(V'(\mathbb{A}_f)^2)$ and that $V'$ is anisotropic. Let $\Gamma'\subset \SO(V')$ be an infinite cyclic subgroup which fixes $\varphi_f'$. Then Theorem \ref{thm:mta}(ii) holds for $c=c(U)$ and we have
    \begin{equation}\label{tsd}
        G(\tau,\varphi_f;c) = \frac{1}{[\Gamma_U:\Gamma']}\sum_{X\in U(\mathbb{Q})^2} \varphi_f''(X) q^{Q(X)} \sum_{T\in \Sym_2(\mathbb{Q})} \beta(T,\varphi_f',v) q^T.
    \end{equation}
\end{lemma}

\begin{proof}
For extra clarity: the first sum on the right of \eqref{tsd} is the theta function of a positive definite quadratic space and hence converges absolutely, and the second sum is the one from Proposition \ref{prp:gsp} with the data $(V',\varphi_f',\Gamma')$.

Let $\mathcal{H}'\subset \mathcal{H}$ be the oriented hyperbolic line corresponding to $U(\mathbb{R})$ which we also identify with a connected component of the Grassmannian of negative oriented lines in $V'(\mathbb{R})$, and let $M'=\Gamma'\backslash \mathcal{H}'$. 

The restriction of $\psi^0(Xa)$ to $\mathcal{H}'$ is given by $\psi'^0(X'a)$ where $X'$ is the component of $X$ in $V'$ and $\psi'^0(X'a)$ denotes the form on $\mathcal{H}'$ defined by the same formula as for $\varphi^0(X)$ but for $V'_{\mathbb{R}}$. It follows that 
\[
    \Theta(T,\varphi_f,\psi^0,v)|_{\mathcal{H}'} = \sum_{X''\in U(\mathbb{Q})^2} \left(\varphi_f''(X'') \sum_{\substack{X'\in V'(\mathbb{Q})^2 \\ Q(X') = T - Q(X'')}} \varphi_f'(X') \psi'^0(X'a) \right).
\]
The inner sum is $\Gamma'$-invariant and hence may be identified with a form on $M'$. Integrating both sides and collecting terms we get
\begin{align*}
    \int_c\Theta(T,\varphi_f,\psi^0,v) &= \frac{1}{[\Gamma_U:\Gamma']}\int_{M'}\Theta(T,\varphi_f,\psi^0,v) \\ &= \frac{1}{[\Gamma_U:\Gamma']}\sum_{\substack{T',T''\in \Sym_2(\mathbb{Q}) \\ T = T' + T''}} r(T'',\varphi_f'') \int_{M'}\Theta'(T',\varphi_f',\psi'^0,v) 
\end{align*}
where for any $T\in \Sym_2(\mathbb{Q})$
\begin{align*}
    r(T,\varphi_f'') &:= \sum_{\substack{X\in U(\mathbb{Q})^2 \\ Q(X) = T}} \varphi_f''(X) \\
    \Theta'(T,\varphi_f',\psi'^0,v) &:= \sum_{X\in V'(\mathbb{Q})^2} \varphi_f'(X) \psi'^0(Xa).
\end{align*}
Replacing $v$ with $vt$ and integrating against $\frac{dt}{t}$ this yields 
\[
    \beta(T,\varphi_f,v;c) = \frac{1}{[\Gamma_U:\Gamma']}\sum_{\substack{T',T''\in \Sym_2(\mathbb{Q}) \\ T = T' + T''}} r(T'',\varphi_f'') \beta(T',\varphi_f',v),
\]
which is the $T$th coefficient of the right-hand side of \eqref{tsd}.
\end{proof}

Now to prove Theorem \ref{thm:mta}(ii), it is obviously sufficient to consider the case where $C$ has one component, so suppose $c=c(U)$. We may write $\varphi_f$ as a linear combination of tensor products $\varphi_f'\otimes\varphi_f''$ where $\varphi_f'\in \mathscr{S}(U^\perp(\mathbb{A}_f)^2)$ and $\varphi_f''\in \mathscr{S}(U(\mathbb{A}_f)^2)$, and furthermore we may find a finite index subgroup $\Gamma'\subset \Gamma_U$ which fixes each $\varphi_f'$. Then Theorem \ref{thm:mta}(ii) holds for each of the series $G^\ast(\tau,\varphi_f'\otimes \varphi_f'';c)$ by Lemma \ref{lem:rdl}, hence Theorem \ref{thm:mta}(ii) also holds for their sum $G^\ast(\tau,\varphi_f;c)$.

\begin{remark}\label{rmk:egd}
Let $L\subset V$ be an integral lattice, as in Section \ref{sec:mts}. Let
\[
    L'' = L\cap U, \hspace{1cm} L' = L\cap U^\perp,
\]
and let $L_0=L''\oplus L'$. From the isomorphism $D_{L''}\oplus D_{L'}\cong D_{L_0}$ we have an isomorphism
\[
    \mathbb{C}[D_{L''}^2]\otimes \mathbb{C}[D_{L'}^2] \cong \mathbb{C}[D_{L_0}^2].
\]
Further, there is a linear map
\[
    \mathbb{C}[D_{L_0}^2]\to \mathbb{C}[D_L^2],\hspace{1cm} \mathfrak{e}_{\mu}\mapsto \sum_{\nu\in (L / L_0)^2} \mathfrak{e}_{\mu+\nu}.
\]
Both maps intertwine the Weil representations of $\Sp_4(\mathbb{Z})$, and so they give an intertwining operator $\iota: \rho_{L''}\otimes \rho_{L'}\to \rho_L$. If $\Gamma'\subset \Gamma_U$ is a finite index subgroup that acts trivially on  $D_{L'}$, then an inspection of the proof shows that
\[
    G^\ast_L(\tau;c) = \frac{1}{[\Gamma_U:\Gamma']}\iota(\Theta_{L''}(\tau)\otimes G_{L'}(\tau)),
\]
where $\Theta_{L''}(\tau)$ is the $\rho_{L''}$-valued theta function attached to the positive definite quadratic lattice $L''$, and $G^\ast_{L'}(\tau)$ is the $\rho_{L'}$-valued function defined by
\[
    G^\ast_{L'}(\tau )_\mu = \sum_{T\in \Sym_2(\mathbb{Q})} \beta(T,\varphi_f'(\mu),v)q^T .
\]
\end{remark}

%--------------------------------------------------------

\subsection{Proof of Theorem \ref{thm:mtb}} \label{sec:mwp}

If $R$ is one of the rings $\mathbb{Q}_p$, $\mathbb{R}$, $\mathbb{A}$, or $\mathbb{A}_f$,  let 
\[
    \omega:\Sp_4(R)\to \Aut(\mathscr{S}(V(R)^2))
\]
be the Weil representation attached to the standard additive character $\psi$ of $R$ \cite[\S 5]{MR1286835}, \cite{MR217026}. We give below the formulas for $\omega$, in this case, when $R$ is $\mathbb{Q}_p$ or $\mathbb{R}$. Since $\omega$ acts multiplicatively on tensor products of Schwartz functions, this is enough to determine $\omega$ over the adeles.

Let $F$ be one of the local fields $\mathbb{Q}_p$ or $\mathbb{R}$. Let $\abs{\cdot}$ denote the absolute value and let $(\;,\;)_F$ be the Hilbert symbol on $F$. Let $m(a) = (\begin{smallmatrix} a & \\ & \tp {a^{-1}} \end{smallmatrix})$ and $n(b)=(\begin{smallmatrix} 1 & b \\ & 1 \end{smallmatrix})$ for $a\in \GL_2(F)$ and $b\in \Sym_2(F)$. For $\varphi\in \mathscr{S}(V(F)^2)$ and $X\in V(F)^2$ we have the formulas 
\begin{align*}
    (\omega(m(a))\varphi)(X) = \;& (d(V),\det a)_F |\det a| \varphi(Xa)  \\
    (\omega(n(b))\varphi)(X) = \; & \psi(\tr(b Q(X)))\varphi(X) \\
    (\omega((\begin{smallmatrix} & - 1 \\ 1 & \end{smallmatrix} ))\varphi)(X) = \;& (-1,d(V))_F \int_{V(F)^2} \varphi(x)\psi( - \tr((X,x)))dx
\end{align*}
where $dx$ denotes the self-dual Haar measure with respect to the pairing $\psi(\tr((\;,\;)))$ on $V(F)^2$.

We further define the finite Weil representation $\rho_L$ of $\Sp_4(\mathbb{Z})$ on the group algebra $\mathbb{C}[(D_L)^2]$. We identify the dual space $\mathbb{C}[D_L^2]^\vee$ with the space of functions $D_L^2\to\mathbb{C}$, which in turn can be identified with the subspace of $\mathscr{S}(V(\mathbb{A}_f)^2)$ consisting of Schwartz functions which are supported on $(L^\#)^2\otimes \widehat{\mathbb{Z}}$ and constant on cosets of $L^2\otimes\widehat{\mathbb{Z}}$. The restriction of $\omega$ to $\Sp_4(\widehat{\mathbb{Z}})$ stabilizes this subspace, and hence defines a representation $\omega_L^\vee$ of $\Sp_4(\widehat{\mathbb{Z}})$ on $\mathbb{C}[(D_L)^2]^\vee$. We let $\omega_L$ be the dual representation of $\omega_L^\vee$, and let $\rho_L$ denote the restriction of $\omega_L$ to $\Sp_4(\mathbb{Z})$.

\begin{proof}[Proof of Theorem \ref{thm:mtb}]
First we construct a Schwartz form $\varphi_f\otimes \varphi_\infty$ on $V(\mathbb{A})^2$. Consider the finite Schwartz function
\[
    \varphi_f = \sum_{\mu\in D_L^2} \varphi_f(\mu) \otimes \mathfrak{e}_\mu \in \mathscr{S}(V(\mathbb{A}_f)^2)\otimes \mathbb{C}[D_L^2].
\]
It has level $\omega_L$ in the sense that
\begin{equation}\label{fwr}
    \omega(k)\varphi_f = \omega_L(k)^{-1} \varphi_f,\hspace{1cm} k\in \Sp_4(\widehat{\mathbb{Z}}).
\end{equation}
This is easily seen from the definions above, by evaluating the dual basis of $\mathfrak{e}_\mu$ on both sides. For $X\in V(\mathbb{R})^2$ let $\varphi_{\infty}(X) = e^{-2\pi \tr Q(X)} \varphi^0(X)$. By Remark \ref{rmk:kmr} $\varphi_\infty$ belongs to $ \mathscr{S}(V(\mathbb{R})^2)\widehat\otimes \Omega^2(\mathcal{H})$ and it has weight $2$ in the sense that
\begin{equation}\label{kmw}
    \omega(k(u))\varphi_{\infty} = \det(u)^2 \varphi_\infty\hspace{1cm} u\in U(2),
\end{equation}
where $k(u) := (\begin{smallmatrix}
    a & -b \\ b & a
\end{smallmatrix})$, for $u=a+ib\in U(2)$, denotes an element of the standard maximal compact subgroup of $\Sp_4(\mathbb{R})$.

Let $g\in \Sp_4(\mathbb{A})$. We consider now the theta series
\[
    \theta_L(g) = \sum_{X\in V(\mathbb{Q})^2} \omega(g)(\varphi_f\otimes \varphi_{\infty})(X)\in \mathbb{C}[(D_L)^2]\otimes \Omega^2(M_\Gamma).
\]
By Poisson summation, \eqref{fwr}, and \eqref{kmw}, it satisfies
\begin{align}\label{tfa}
    \theta_L(\gamma g) &= \theta_L(g)&\gamma\in \Sp_4(\mathbb{Q}),  \\ \label{tfl}
    \theta_L(gk) &= \omega_L(k)^{ - 1} \theta_L(g) &k\in \Sp_4(\widehat{\mathbb{Z}}), \\ \label{tfw}
    \theta_L(gk(u)) &= \det(u)^2 \theta_L(g) &u\in U(2).
\end{align}
Let $\tau=u+iv\in \mathfrak{H}$, let $a\in \GL^+_2(\mathbb{R})$ be any matrix satisfying $v = a\tp{a}$, and let $g_\tau = n(u) m(a)\in \Sp_4(\mathbb{R})$. Then by the formulas for $\omega$ we have
\[
    (\omega(g_\tau)\varphi_{\infty})(X) = (\det a)^{2} e^{2\pi i \tr(u Q(X))}\varphi_\infty(Xa) = \det (v)\varphi^0(Xa) q^{Q(X)}.
\]
Thus
\begin{equation}\label{kmg}
    \det(v)^{-1}\theta_L(g_\tau) =\sum_{T\in \Sym_2(\mathbb{Q})} \Theta(T,\varphi_f,\varphi^0,v) q^T.
\end{equation}
where the definition of $\Theta(T,\varphi_f,\varphi^0,v)$ from previously is extended linearly to the $\rho_L$-valued Schwartz function $\varphi_f$. Then by the same argument with the transgression formula as in the proof of Theorem \ref{thm:mta}(i) we have 
\[
    \det (v)^{ -1}\int_s\theta_L(g_\tau) = I_L(\tau;s) + G_L^\ast(\tau;c).
\]
It follows from \eqref{tfa}, \eqref{tfl}, and \eqref{tfw}, that the function  $\det (v)^{ -1}\theta_L(g_\tau)$ on $\mathfrak{H}$ is modular of weight $2$ and 
level $1$. Consequently $(\det v)^{ -1}\int_s\theta_L(g_\tau)$ is also modular of weight $2$ and level $1$.
\end{proof}

%--------------------------------------------------------

\subsection{Proof of Theorem \ref{thm:mtc}}

Throughout this section $T$ will be a positive definite matrix. To prove Theorem \ref{thm:mtc}, the main difficulty will consist of proving the following.

\begin{proposition}\label{prp:bbb}
    For any $\varphi_f\in \mathscr{S}(V(\mathbb{A}_f)^2)$ and $\varepsilon>0$ we have
    \[
        |\beta(T,\varphi_f,T^{ - 1};c)| \ll \det (T)^{3/2 + \varepsilon}.
    \]
\end{proposition}
Let us first see how Proposition \ref{prp:bbb} implies Theorem \ref{thm:mtc}. Let
\[
F_L(\tau;s) := I_L(\tau;s)+G_L^\ast(\tau;c).
\]

\begin{lemma}\label{lem:hbb}
    For every $\mu\in D_L^2$
    \[
        |F_L(\tau;s)_\mu| \ll \det (v)^{-1}.
    \]
\end{lemma}
\begin{proof}
By a compactness argument we may assume that $S$ is simply connected and has globally defined coordinates $x:S\to \mathbb{R}^2$. In particular we may choose a lift $\tilde s:S\to \mathcal{H}$ of $s$.

Select an $\Sp_4(\mathbb{Z})$-equivariant inner product $\nm{\cdot}$ on $\mathbb{C}[D_L^2]$. By Theorem \ref{thm:mtb} it follows that the function $\nm{F_L(\tau;s)}\det v=\nm{\int_s \theta_L(g_\tau)}$ is $\Sp_4(\mathbb{Z})$-invariant. Thus it suffices to consider this function on the standard Siegel fundamental domain, and hence it suffices to show that $\nm{\int_s \theta_L(g_\tau)}\to 0$ uniformly as $\det v\to \infty$ \cite[$2.12_2$]{freitag2013siegelsche}. Using the same notation as in \eqref{pff}, we get
\[
    \tilde s^\ast \omega(g_\tau)\varphi_\infty(X) = \det(v)^{1 /2}f(X,x) e^{ - 2\pi \tr Q_{\tilde s(x)}(X)v}e^{2\pi i \tr Q(X)u} dx
\]
Pick a positive definite form $q$ satisfying $q< \tr Q_{\tilde s(x)}$ for all $x\in S$. Then we obtain
\[
    \nm{\int_s \theta_L(g_\tau)} \ll \det(v)^{1 /2}\sum_{X\in (L^\#)^2} \sup_{x\in S} |f(X,x)| e^{ -2\pi \tr q(X)v}
\]
As $\tr q(X)v \geq 2 \det(q(X)v)^{1 /2}$ the series converges termwise to $0$ as $\det v\to \infty$. Since $\tr q(X)v \gg \tr q(X)$ the result follows by dominated convergence \cite[2.12]{freitag2013siegelsche}. 
\end{proof}

Now write $\alpha_L(T,v;s)_\mu$ for the $T$th Fourier coefficient of $F_L(\tau;s)$. By Lemma \ref{lem:hbb} we obtain the Hecke bound $|\alpha_L(T,T^{-1};s)_\mu| \ll \det T$. Combining this with Proposition \ref{prp:bbb} completes the proof of Theorem \ref{thm:mtc}.

We turn to the proof of Proposition \ref{prp:bbb}. By the same argumentation as in Subsection \ref{sec:tap}, it suffices to consider a tensor product $\varphi_f =\varphi_f'\otimes \varphi_f''$. Then by \eqref{ebf}
\begin{equation}\label{bef}
    \beta(T,\varphi_f,v;c) = 
    \sum_{
    \substack{
        T = T' + T'' \\
        T'\;\text{indefinite} \\
        T''\;\text{pos. def.}
    }} r(T'',\varphi_f'') \rho(T',\varphi_f') W^\ast\left(\tr(T' v), 4\abs{\det(T' v)}\right).
\end{equation}
We fix an integer $N>0$ such that $Q(X)\varphi_f(X)\in \frac{1}{N}\Sym_2(\mathbb{Z})$ for all $X\in V(\mathbb{Q})^2$. Hence $\beta(T,\varphi_f,v;c)=0$ if $T\not\in \frac{1}{N}\Sym_2(\mathbb{Z})$. We write
\[
    T = \frac{1}{M} \begin{pmatrix} n & r /2 \\ r /2 & m \end{pmatrix}, \hspace{1cm} D = r^2-4nm,
\]
with $M|N$ and $n,r,m$ are coprime, and
\[
    T' = \frac{1}{N}\sigma, \hspace{1cm} \sigma = \begin{pmatrix} \sigma_1 & \sigma_0 \\ \sigma_0 & \sigma_2 \end{pmatrix},
\]
with $\sigma_1,\sigma_2,2\sigma_0\in \mathbb{Z}$. By simple computations
\[
    \det(T-T') = \det T -\tr(T'T^{-1})\det T+\det T',
\]
\[
    \tr(T' T^{ - 1}) = 
    \frac{4M}{N|D|} (\sigma_1 m - \sigma_0r + \sigma_2n) ,
\]

\[
    \det(T') = -\frac{\sigma_0^2-\sigma_1\sigma_2}{N^2}, \hspace{1cm} \det T = \frac{ |D|}{4M^2},
\]
Hence by Lemma \ref{lem:rnb}, and the same estimate for $r(T'',\varphi_f'')$, we obtain
\[
    \abs{\beta(T,\varphi_f,T^{ - 1};c)} 
    \ll 
    \sum_{d>0}\sum_{t< \frac{N|D|}{2M}-\frac{Md}{2N}} 
    \mathcal{N}_{T}(t,d)\; 
    C_T(t,d)^\varepsilon\;
    W^\ast \left(\frac{2Mt}{N|D|},\frac{4M^2d}{N^2|D|}\right),
\]
where
\[
    \mathcal{N}_{T}(t,d) = \#\left\{\sigma \in \Sym_2(\mathbb{Z}): \sigma_1 m - \sigma_0r + \sigma_2n =\frac{t}{2}, \sigma_0^2 - \sigma_1\sigma_2 = \frac{d}{4}\right\},
\]
\[
    C_T(t,d) = \left(\frac{ |D|}{4M^2} - \frac{t}{2NM}-\frac{d}{4N^2} \right)\frac{d}{4N^2}.
\]
Now the task at hand is to provide estimates for $\mathcal{N}_T(t,d)$. We do so through two lemmas. First, let
\[
    \mathfrak{L}_T = \{(u,v,w)\in \mathbb{Z}^3 : ru+mv+nw=0\},
\]
and for a positive integer $A$ let
\[
    r_T(A)= \#\{(u,v,w)\in \mathfrak{L}_T:u^2-vw=A \}.
\]
\begin{lemma}\label{lem:ltd}
    The quadratic form $(u,v,w)\mapsto u^2-vw$ on $\mathfrak{L}_T$ is positive definite of rank $2$ and has discriminant $D$. In particular $r_T(a) <\infty$.
\end{lemma}
\begin{proof}
This follows from Lemma \ref{lem:lqt} in the next section, applied to the binary quadratic form $q(x,y)=nx^2+rxy+my^2$ and using the isomorphism of quadratic lattices
\[
    \mathbb{Z}^3\cong L',\hspace{1cm} (u,v,w)\mapsto \begin{pmatrix}
        -u & w \\ -v & u
    \end{pmatrix}.
\]
%Consider the homomorphism
%\[
%    B:\mathbb{Z}^3\to \mathbb{Z} / \abs{D}\mathbb{Z}, \hspace{1cm} (u,v,w)\mapsto ru+mv+nw \bmod \abs{D}.
%\]
%Clearly, $\mathfrak{L}_T\subset \ker B$ and $(-r,2n,2m)\in \ker B$. Conversely, if $(u,v,w)\in \ker B$, then
%\[
%    (u,v,w) +\frac{ru+mv+nw}{D}(-r,2n,2m)\in \mathfrak{L}_T,
%\]
%hence we obtain the orthogonal decomposition $\ker B = \mathfrak{L}_T\oplus \mathbb{Z}(-r,2n,2m)$. Since $r$, $m$, and $n$ are coprime it follows that $B$ is surjective, and hence $[\mathbb{Z}^3:\ker B]=\abs{D}$. 
%With respect to the quadratic form, $\mathbb{Z}^3$ has discriminant $-2$ and signature $(2,1)$, so the above implies that $\ker B$ has discriminant $-2D^2$ and signature $(2,1)$. Since $\mathbb{Z}(-r,2n,2m)$ has discriminant $2D$ and signature $(0,1)$, it follows that $\mathfrak{L}_T$ has discriminant $D$ and signature $(2,0)$. %We use the signed discriminant , for a quadratic lattice $(\mathfrak{L},Q_\mathfrak{L})$ of rank $n$
%\[
%    \disc \mathfrak{L} = (-1)^{\frac{n(n-1)}{2}} \det ((\lambda_i,\lambda_j)_\mathfrak{L})_{i,j=1,\ldots,n}
%\]
%for a basis $\lambda_1,\ldots,\lambda_n$ of $\mathfrak{L}$.
\end{proof}

\begin{lemma}\label{lem:nre}
    If $t^2\not\equiv Dd\bmod 4$, then $\mathcal{N}_T(t,d)=0$; and if $t^2\equiv Dd\bmod 4$, then $\mathcal{N}_T(t,d)\leq r_T\left(\frac{t^2-Dd}{4}\right)$.
\end{lemma}
\begin{proof}
Consider the linear map
\[
    \Sym_2(\mathbb{Z}) \to \mathfrak{L}_T, \hspace{1cm} \sigma\mapsto (m\sigma_1 - n\sigma_2,2n\sigma_0 -r\sigma_1,r\sigma_2 - 2m\sigma_0).
\]
The kernel is spanned by $(\begin{smallmatrix}
    n & r/2 \\ r/2 & m
\end{smallmatrix})$, since this is a primitive vector. It follows that the map
\[
    \left\{
    \sigma\in\Sym_2(\mathbb{Z}): \sigma_1m-\sigma_0r+\sigma_2n=\frac{t}{2}
    \right\}\to \mathfrak{L}_T
\]
is injective, since if $\sigma$ and $\sigma'$ had the same image, then $\sigma-\sigma'= x(\begin{smallmatrix}
    n & r/2 \\ r/2 & m
\end{smallmatrix})$ for some $x\in \mathbb{Z}$, which would imply that $\frac{t}{2}-\frac{t}{2}=x\frac{-D}{2}$ and hence $x=0$. 

Thus the lemma follows from the polynomial identity
\[
    (m\sigma_1 - n\sigma_2)^2 -(2n\sigma_0-r\sigma_1)(r\sigma_2 - 2m\sigma_0) = (\sigma_1 m - \sigma_0r + \sigma_2n)^2 - (r^2 - 4nm)(\sigma_0^2 - \sigma_1\sigma_2)
\]
\end{proof}
We change the order of summation by labeling $A=\frac{t^2-Dd}{4}$. Together with Lemma \ref{lem:nre}, this leaves us to estimate
\begin{equation}\label{psr}
    \sum_{A>0} r_T(A)\sum_{t,d} C_T(t,d)^\varepsilon \:W^\ast \left(\frac{2Mt}{N|D|},\frac{4M^2d}{N^2|D|}\right).
\end{equation}
The inner sum is over pairs $(t,d)\in\mathbb{Z}^2$ that satisfy $d>0$, $t< \frac{N|D|}{2M}-\frac{Md}{2N}$, and $t^2-Dd=4A$. Clearly, $d$ is determined by $t$, and we see that such pairs are in bijection with $t\in\mathbb{Z}$ that satisfy
\[
    -2\sqrt{A}<t<\min\left\{\frac{N|D|}{M}-2\sqrt{A},2\sqrt{A}\right\}, \hspace{1cm} t^2\equiv 4A\mod\abs{D}.
\]
In particular, the number of such $t$ is bounded by $\abs{D}^\varepsilon$. 

We have $C_T(t,d) \ll D^2(1+\sqrt{A}/\abs{D})^3$, and by \eqref{Wfe} and $t< \frac{N\abs{D}}{M}-2\sqrt{A}$
\[
    W^\ast \left(\frac{2Mt}{N|D|},\frac{4M^2d}{N^2|D|}\right) \ll e^{-16\pi M\sqrt{A} / (N\abs{D})}.
\]
Taken together, this yields a bound of \eqref{psr} by
\[
    \abs{D}^{3\varepsilon} \sum_{A>0}r_T(A)e^{-(16\pi M/N - 3\varepsilon)\sqrt{A/D^2}}
\]
By the well-known estimate $\sum_{A<x} r_T(A) \ll x /\sqrt{\abs{D}}$ from the theory of binary quadratic forms (this is where we use that $\mathfrak{L}_T$ has discriminant $D$) we obtain the desired bound by $\abs{D}^{3/2+3\varepsilon}$.

\begin{remark}\label{rem:eeo}
    Pick $a\in\GL_2^+(\mathbb{R})$ such that $T^{-1} = a \tp{a}$. Relabeling $\sigma = (T')^{a}$, we obtain from \eqref{bef}
    \[
        \abs{\beta(T,\varphi_f,T^{-1};c)}\ll\det (T)^{2\varepsilon}\sum_{\substack{\sigma\in\frac{1}{N}\Sym_2(\mathbb{Z})^a \\ \sigma\;\text{indefinite} \\ 1-\sigma\;\text{pos. def. }}} \det(1-\sigma)^\varepsilon \det(\sigma)^\varepsilon \:W^\ast(\tr\sigma,4\abs{\det\sigma}).
    \]
    Since $\frac{1}{N}\Sym_2(\mathbb{Z})^a\subset\Sym_2(\mathbb{R})$ is a lattice with covolume $N^3\det(T)^{3/2}$, it would be expected heuristically that the sum above is $O(\det(T)^{3/2})$. We tried looking for a proof along these lines, using Poisson summation, but ultimately we found it easier to use the more direct methods above. Nonetheless, this suggests that the exponent in Proposition \ref{prp:bbb} is optimal, barring a more explicit expression to replace the $\varepsilon$.
\end{remark}

%--------------------------------------------------------

\section{Example}\label{sec:exa}

In this section, we demonstrate how to compute $\iota_L(T;s)_0$ in a very explicit example. In particular, we see that $I_L(\tau;s)\neq0$.

%--------------------------------------------------------

\subsection{The surface $M'$}

We consider the quadratic space $V$ and lattice $L$ in Example \ref{exa:ile}, \ref{exa:sce}, and \ref{exa:mte}. We let $\Gamma\subset\PSL_2(\mathbb{Z}[i])$ be the congruence subgroup $\Gamma=\Gamma_1(2-i)$, consisting of elements $\gamma\in \PSL_2(\mathbb{Z}[i])$ with $\gamma\equiv \pm(\begin{smallmatrix}
    1 & \ast \\ 0 & 1
\end{smallmatrix})\bmod (2-i)$. The manifold $M$ can also be realized as the complement of the $(-2,3,8)$-Pretzel link in $S^3$, as one can easily check that $\Gamma$ coincides with the index $12$ subgroup of $\PSL_2(\mathbb{Z}[i])$ generated by $(\begin{smallmatrix}
    1 & 1 \\ & 1
\end{smallmatrix})$, $(\begin{smallmatrix}
    1 & i \\ & 1
\end{smallmatrix})$, and $(\begin{smallmatrix}
    1 &  \\ 2-i & 1
\end{smallmatrix})$, see \cite{MR1020042}.

Consider the quadratic space and lattices
\[
    V' = \left\{\begin{pmatrix} - x & y \\ z & x \end{pmatrix} : x,y,z\in \mathbb{Q} \right\},\hspace{1cm}
    L' = \left\{\begin{pmatrix} - x & y \\ z & x \end{pmatrix} : x,y,z\in \mathbb{Z} \right\}.
\]
of rank $3$. These are the orthogonal complement of the vector $(\begin{smallmatrix}
    i & \\ & i
\end{smallmatrix})\in L$. It is easy to see that the stabilizer of $(\begin{smallmatrix}
    i & \\ & i
\end{smallmatrix})$ in $\Gamma$ is equal to the congruence subgroup $\Gamma':=\Gamma_1(5)\subset \PSL_2(\mathbb{Z})$. Further, the hyperbolic plane corresponding to this vector is given by
\[
    \mathcal{H}^2 = \{x+jy: x\in\mathbb{R},y\in \mathbb{R}_{>0}\}
\]
inside hyperbolic $3$-space, which we now denote $
\mathcal{H}^3$ for emphasis. We let $M' = \Gamma'\backslash \mathcal{H}^2 $. Thus, $M'$ is an oriented immersed arithmetic hyperbolic surface in $M$.

Let $s':S\to M'$ be a smooth map, where $S$ is an oriented compact surface. For a fixed curve $c'$ in $M'$ such a map $s'$ with boundary $c'$ exists if and only if the homology class $[c']\in H_1(M')$ is trivial. This follows by applying Lemma \ref{lem:eog} to $M'\times S^1$ and then projecting the resulting map to $M'$. We can then apply the same techniques as in Section \ref{sec:tln} and \ref{sec:mgs} for $V'$ and $L'$ to construct generating series
\[
    I_{L'}(\tau;s')_0 =  \sum_{T\in \Sym_2(\mathbb{Q})} \iota_{L'}(T;s')_0 q^T.
\]
The numbers $\iota_{L'}(T;s')_0$ can be interpreted as the intersection number of $s'$ with the special $0$-cycle corresponding to $T$, or as the winding number of $c'=\partial s'$ and this $0$-cycle. Let $V''=\{(\begin{smallmatrix} ix & \\ & ix\end{smallmatrix}):x\in \mathbb{Q}\}$ and let $L''\subset V''$ be the lattice given by $x\in \mathbb{Z}$. Then we have orthogonal decompositions $V=V'\oplus V$ and $L=L'\oplus L''$. Let $s$ be the composite map $S\to M'\to M$. Then by the same argument as in Lemma \ref{lem:rdl}, we have the identity
\begin{equation}\label{iip}
    I_L(\tau;s)_0 = \Theta(\tau)_0 I_{L'}(\tau;s')_0,
\end{equation}
where the zero subscripts indicate that we consider the components at the zero element of the discriminant groups, and 
\[
    \Theta(\tau)_0 = \sum_{n,m\in \mathbb{Z}} q^{\left(\begin{smallmatrix} n^2 &nm \\ nm & m^2\end{smallmatrix}\right)} \hspace{1cm} \tau  \in \mathfrak{H}
\]
is the theta function attached to $L''$.

%--------------------------------------------------------

\subsubsection{Geodesics in $M'$} 

From now on we identify $\mathcal{H}^2$ with the Poincare upper half-plane in $\mathbb{C}$. A fundamental domain for $\Gamma'$ is depicted in Figure \ref{fig:fd}. It is given by the hyperbolic polygon $\mathcal{F}$ with vertices $\infty$, $0$, $1/3$, $2/5$, $1/2$, and $1$ on the boundary of $\mathcal{H}^2$. 

\begin{figure}[ht]
    \centering
\includegraphics{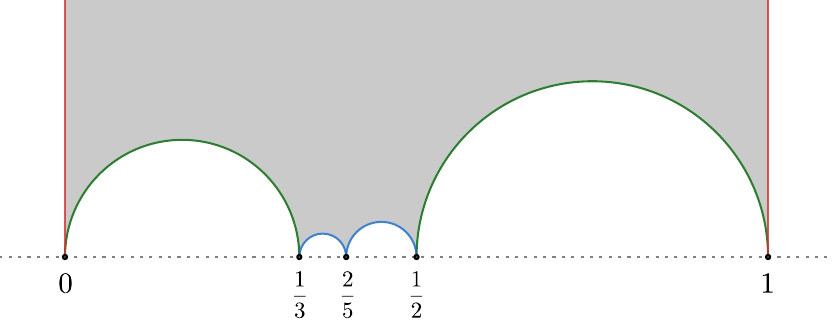}
    \caption{Fundamental domain for $\Gamma'$.%, and a tesselation by fundamental domains for $\PSL_2(\mathbb{Z})$
    }
    \label{fig:fd}
\end{figure}

There are four cusps of $M'$, which are given by
\[
    x_1 = [\infty], \hspace{1cm} x_2=[0]=[1], \hspace{1cm} x_3 = [1/3]=[1/2], \hspace{1cm} x_4 = [2/5].
\]
Since the modular curve $X_1(5)$ has genus $0$, it follows that $M'$ is homeomorphic to a $2$-sphere with $4$ punctures, and in particular $H_1(M')$ is free of rank $3$ and $H_2(M')=0$.

We denote by $\ell_{z_1,z_2}$ the oriented geodesic connecting $z_1$ and $z_2$. Let $c_r$, $c_g$, and $c_b$ denote the geodesics in $M'$ obtained as the images of $\ell_{\infty,0}$, $\ell_{0,1/3}$, and $\ell_{1/3,2/5}$. Thus these are just the images of the sides of $\mathcal{F}'$. The following lemma will be used to compute the homology class of geodesics on $M'$.
\begin{lemma}
    For a closed oriented curve $c$ in $M'$ let $r$, $g$, and $b$ denote the intersection pairing of $c$ with $c_r$, $c_g$, and $c_b$, respectively. Then the map $c\mapsto (r,g,b)$ gives an isomorphism $H_1(M')\cong \mathbb{Z}^3$.
\end{lemma} 
\begin{proof}
The homology group $H_1(M')$ is generated by loops $\ell_i$ around the cusps $x_i$ with one relation which stipulates that their sum is trivial. In particular, $[\ell_1]$, $[\ell_2]$, and $[\ell_3]$ provide a basis for $H_1(M')$, and with respect to this basis the map in the lemma has matrix
\[
    \begin{pmatrix} 
        1 & -1 & \\
        & 1 & -1 \\
        & & 1
     \end{pmatrix} .
\]
\end{proof}
This gives an efficient algorithm to compute the homology class of a geodesic $c'$ in $M'$ attached to the $\Gamma'$-equivalence class of an indefinite binary quadratic form $q$ with discriminant $d>0$. 

Define the function
\[
    h:\mathbb{R} - \left\{0,\frac{1}{3},\frac{2}{5},\frac{1}{2},1\right\}\to \{-1,0,1\}^3
\]
by 
\[
    h(\rho) = \begin{cases} 
        ( 1,0,0) & \rho \in ( -\infty,0) \\
        (0, 1,0) & \rho\in (0,\frac{1}{3}) \\
        (0,0, 1) & \rho \in (\frac{1}{3},\frac{2}{5}) \\
        (0,0, -1) & \rho\in (\frac{2}{5},\frac{1}{2}) \\
        (0, -1,0) & \rho\in (\frac{1}{2},1) \\
        ( -1,0,0) & \rho\in (1,\infty)
    \end{cases} .
\]
Let $q_i(x,y)=a_ix^2+b_ixy+c_iy^2$, $i=1,\ldots,\ell$ be the list of forms in the $\Gamma'$-class of $q$ for which the associated geodesic intersects $\mathcal{F}$. We let
\[
    \rho_i^\pm = \frac{ - b_i\pm \sqrt{d}}{2a_i}.
\]
If $a_i=0$ or $\rho_i^\pm\in \left\{0,\frac{1}{3},\frac{2}{5},\frac{1}{2},1\right\}$ for some $i$, then $c'$ will not be closed, so we assume that this is not so. Then let
\[
    h_i = \frac{1}{2}(h(\rho_i^+) - h(\rho_i^- ))
\]
\begin{lemma}\label{lem:ghl}
    The homology class of $c'$ is given by $h=h_1+\cdots+h_\ell$.
\end{lemma}
\begin{proof}
Let $z_i^+$ and $z_i^-$ be the intersection points of the geodesic attached to $q_i$ and the boundary of  $\mathcal{F}$, such that $\rho_i^{+}$, $z_i^+$, $z_i^-$, and $\rho_i^-$ lie on the geodesic in this order. For each of the boundary geodesics $\ell_{x_1,x_2}$ of $\mathcal{F}$ we see that $z_i^\pm\in \ell_{x_1,x_2}$ if and only if $\rho_i^\pm\in (x_1,x_2)$. Now we order the $q_i$ such that $z_i^-$ is in the same $\Gamma'$-orbit as $z_{i+1}^+$, and let $y_i\in M'$ be this orbit. It follows that $y_1,\ldots,y_\ell$ are precisely the intersection points of $c'$ and $c_r$, $c_g$, $c_b$. Furthermore $h(\rho_{i+1}^+)$ gives the orientation number of $c'$ with $(c_r,c_g,c_b)$ at $y_i$, and $h(\rho_{i}^-)$ equals minus the orientation number. The lemma now follows by rearranging the sum as follows
\[
    h_1 +\cdots +h_{\ell} = \sum_{i = 1}^\ell \frac{1}{2}( h(\rho_{i + 1}^+) - h(\rho_{i}^- )) .
\]
\end{proof}

The proof of Lemma \ref{lem:ghl} also indicates how to compute the forms $q_1,\ldots,q_{\ell}$ at the same time as their homology class $h\in \mathbb{Z}^3$. Given $q_i$, we compute $\rho_{i}^\pm$ and determine which of the geodesics $c_r$, $c_g$, or $c_b$, that $y_i$ lies on. Then $q_{i+1}$ is given by applying to $q_i$ the element $\gamma_i\in \Gamma'$ which carries the side of $\mathcal{F}$ containing $z_i^-$ to the side containing $z_{i+1}^+$.

We have singled out three specifically chosen geodesics $c_1'$, $c_2'$, and $c_3'$, see Table \ref{tab:ght}. One important reason for the choice of these geodesics is that their homology classes sum to zero. In particular, if we take $c'$ to be their disjoint union then the homology class of $c'$ is trivial, and there exists a map $s'$ with boundary $c'$.

\begin{table}[h]
    \centering
    \begin{tabular}{c  c c  c}
         & $q$ & $\ell$ & $h$ \\[.5ex]\hline \\[-2ex]
        $c_1'$ & $x^2-3y^2$ & $5$ & $(3,2,0)$ \\
        $c_2'$ & $2x^2+2xy-y^2$ & $4$ & $(2,-1,-1)$ \\
        $c_3'$ & $-3x^2-11xy+9y^2$ & $9$ & $(-5,-1,1)$
    \end{tabular}
    \caption{Three geodesics in $M'$}
\label{tab:ght}
\end{table}

%--------------------------------------------------------

\subsection{The $0$-cycle $c(T)$}

For a pair of vectors $X\in V'(\mathbb{R})^2$ for which $Q(X)$ is positive definite, we let $c_X\in \mathcal{H}^2$ denote the associated oriented point. This is oriented by the same rule as in Section \ref{sec:tst}. We let $c(X)\in M'$ denote the image of $c_X$, which only depends on the $\Gamma'$-orbit of $X$. For a positive definite matrix $T\in \Sym_2(\mathbb{Z})$ we define the special cycle
\[
    c(T) := \sum_{\substack{X\in (L')^2 / \Gamma' \\ Q(X)=T}} c(X).
\]
We shall give a more explicit description of $c(T)$ in this subsection. 

If $q(x,y)=ax^2+bxy+cy^2$ is a definite binary quadratic form with real coefficients, we let $z_q\in \mathcal{H}^2$ denote the unique root of $q(z,1)=0$. We take $z_q$ to be positively (resp. negatively) oriented if $q$ is positive (resp. negative) definite.

We let $\PSL_2(\mathbb{R})$ act on the set of binary quadratic forms with real coefficients from the right in the usual way
\[
    (q\cdot (\begin{smallmatrix} a & b \\ c & d \end{smallmatrix}))(x,y) = q(ax + by, cx +yd).
\]

\begin{lemma}\label{lem:qii}
    Let $X_i = (\begin{smallmatrix} -x_i & y_i \\ z_i & x_i \end{smallmatrix})\in V'(\mathbb{R})$ for $i=1,2$ and let $X=(X_1,X_2)$. Let $q(x,y)=ax^2+bxy+cy^2$ be the binary quadratic form with coefficients
    \[
        a = \det \begin{pmatrix} x_1 & z_1 \\ x_2 & z_2 \end{pmatrix}\hspace{1cm} b = \det \begin{pmatrix} z_1 & y_1 \\ z_2 & y_2 \end{pmatrix}\hspace{1cm} c = \det\begin{pmatrix} x_1 & y_1 \\ x_2 & y_2 \end{pmatrix}.
    \]
    Then 
    \begin{enumerate}[label = (\roman*), font = \upshape]
        \item the discriminant of $q$ is equal to $-4\det Q(X)$;
        \item if $Q(X)$ is positive definite we have $c_X = z_q$;
        \item the mapping $X\mapsto q$ is $\PSL_2(\mathbb{R})$-equivariant in the sense that if $q$ is associated with $X$ then $q\cdot g$ is associated with $g^{-1}X$ for $g\in \PSL_2(\mathbb{R})$;
        \item the mapping $X\mapsto q$ is $\GL_2(\mathbb{R})$-equivariant in the sense that if $q$ is associated with $X$ then $\det(h)q$ is associated with $Xh$ for $h\in \GL_2(\mathbb{R})$.
    \end{enumerate}
\end{lemma}
\begin{proof}
(i), (iii), and (iv) are very simple algebraic manipulations. 

For (ii) note that since $g c_X = c_{gX}$ and $g z_q = z_{g\cdot q}$ for $g\in \PSL_2(\mathbb{R})$, it suffices to consider a single choice of $X$. We take $X = ((\begin{smallmatrix} -1 & \\ & 1 \end{smallmatrix}), (\begin{smallmatrix}  & 1 \\ 1 & \end{smallmatrix}))$. Then the orthogonal compliment of $X$ is spanned by the vector $(\begin{smallmatrix}  & 1 \\ -1 &  \end{smallmatrix})=X(i)$, hence $c_X$ equals $i$ as a point. Furthermore, an easy computation shows that $X_1$ and $X_2$ equal the $x$ and $y$ partial derivatives of the map $z\mapsto X(z)$ at $z=i$. Hence the standard orientation of $\mathcal{H}^2$ agrees with the orientation of the normal bundle of $c_X$, and $c_X$ is positively oriented. On the other hand, the quadratic form associated to $X$ is $q(x,y) = x^2+y^2$, and so $z_q = i$ with positive orientation.
\end{proof}

Let $x(q)\in M'$ denote the image of $z_q$. By Lemma \ref{lem:qii}(iii) the point $x(q)$ only depends on the $\Gamma'$-orbit of $q$, and if $q$ is the quadratic form associated with $X$ then $c(X)=x(q)$. 

For $d\in \mathbb{Z}$ let $\mathcal{Q}_d$ denote the set of integral binary quadratic forms with discriminant $d$, and for $T\in \Sym_2(\mathbb{Z})$ let $L'(T)=\{X\in (L')^2 : Q(X)=T\}$. Assume that $T$ is positive definite and let $d=-4\det T$. Then we obtain an equality of $0$-cycles
\begin{equation}\label{zce}
    c(T) = \sum_{X\in L'(T)/\Gamma' } c(X) = \sum_{q\in \mathcal{Q}_d / \Gamma'} m(T,q)x(q),
\end{equation}
where $m(T,q)$ is the number of $X\in L'(T)$ for which $q$ equals the quadratic form associated to $X$. 

The cycle $c(T)$ has an interesting symmetry. Consider the involutions of $L'(T)$ and $\mathcal{Q}_d$ given by $(\begin{smallmatrix} -x_i & y_i \\ z_i & x_i \end{smallmatrix})\mapsto (\begin{smallmatrix} x_i & y_i \\ z_i & -x_i \end{smallmatrix})$ and $ax^2+bxy+cy^2\mapsto -ax^2+bxy-cy^2$. Let us denote these involutions by $X\mapsto \widetilde X$ and $q\mapsto \tilde q$. If $q$ is associated with $X$ then $\tilde q$ is associated with $\widetilde X$, and the points $x(q)$ and $x(\tilde q)$ are of opposite orientations. In addition
\begin{equation}\label{eq:mst}
    m(T,q) = m(T,\tilde q).
\end{equation}
From Lemma \ref{lem:qii} we also deduce that for $\gamma\in \PSL_2(\mathbb{Z})$
\begin{align}\label{eq:tev}
m(T^\gamma,q) = m(T,q) \\\label{eq:qev}
m(T,q\cdot \gamma) = m(T,q)
\end{align}
and that
\begin{equation}\label{eq:odm}
    m\left(\begin{pmatrix}
        t_1 & t_0 \\ t_0  & t_2
    \end{pmatrix},q\right) = m \left(\begin{pmatrix}
        t_1 & -t_0 \\ -t_0  & t_2
    \end{pmatrix},-q\right).
\end{equation}

%--------------------------------------------------------

\subsubsection{A formula for $m(T,q)$}

Suppose now that $d$ is a negative discriminant. Let $\mathcal{Q}_d^+\subset \mathcal{Q}_d$ be the subset of positive definite forms. Let $\mathcal{C}_d = \mathcal{Q}_d^+ / \PSL_2(\mathbb{Z})$ and let $\mathcal{C}_d^0 \subset \mathcal{C}_d$ be the form class group of discriminant $d$, consisting of classes of primitive forms. For a $q\in \mathcal{Q}_d^+$ we let $[q]\in \mathcal{C}_d$ be its class, and we write the group law on $\mathcal{C}_d^0$ additively.

For $T = (\begin{smallmatrix} t_1 & t_0 \\ t_0 & t_2\end{smallmatrix})\in \Sym_2(\mathbb{Z})$ with $d=-4\det T$, we define $t\in \mathcal{Q}_d^+$ by 
\[
    t(x,y)=t_1x^2+2t_0xy+t_2y^2 
\]
and we write $m(t,q)=m(T,q)$. By \eqref{eq:mst} it is enough to consider positive definite $q$, and by \eqref{eq:tev} and \eqref{eq:qev}, $m(t,q)$ only depends on the classes $[t]$ and $[q]$. For $t$ and $q$ positive definite forms, not necessarily of the same discriminant, we let
\[
    r_q^+(t) := \# \{ h\in M_2(\mathbb{Z}):t = q\cdot h,\; \det(h)>0\}.
\]
Similarly this only depends on the classes $[t]$ and $[q]$.
\begin{proposition}\label{prp:mtq}
    Let $t,q\in \mathcal{Q}_d^+$, let $g$ be the content of $q$, and let $q_0=\frac{1}{g}q$. Then
    \[
        m(t,q) = r_{-2[q_0]}^+(t).
    \]
\end{proposition}
Proposition \ref{prp:mtq} will follow from two technical lemmas about binary quadratic forms. There is a one-to-one correspondence between $\mathcal{C}_d$ and isometry classes of oriented quadratic $\mathbb{Z}$-lattices of rank 2 and discriminant $d$, which is given by sending a quadratic $\mathbb{Z}$-lattice $(\mathfrak{L},Q_{\mathfrak{L}})$ to the binary quadratic form 
\[
    q(x,y) = Q_{\mathfrak{L}}(x\lambda_1 + y\lambda_2)
\]
for a positively oriented basis $\lambda_1,\lambda_2$ for $\mathfrak{L}$. We let $[\mathfrak{L}]\in \mathcal{C}_d$ denote the class corresponding to $(\mathfrak{L},Q_{\mathfrak{L}})$. Here we are using the the signed discriminant of $\mathfrak{L}$ which equals $-4\det Q_{\mathfrak{L}}(\lambda)$, for any basis $\lambda=(\lambda_1,\lambda_2)$. 
This recovers the classical bijection between $\mathcal{C}_d$ and ideal classes in the imaginary quadratic order $\mathcal{O}_d$ of discriminant $d$ by viewing an ideal $\mathfrak{a}$ as a rank 2 lattice with quadratic form given by $x\mapsto N_{\mathbb{Q}(\sqrt{d}) / \mathbb{Q}}(x) / N(\mathfrak{a})$, where $N(\mathfrak{a})$  denotes the norm of $\mathfrak{a}$, and with orientation coming from a fixed choice of embedding $\mathbb{Q}(\sqrt{d}) \subset \mathbb{C}$.

For a positive definite binary quadratic form $q(x,y)=ax^2+bxy+cy^2$ with discriminant $d$ let
\[
    L_q = \{X\in L' : (X,X_q) = 0 \}, \hspace{1cm} X_q = \begin{pmatrix} - b & -2c \\ 2a & b\end{pmatrix}.
\]
Since $L_q\otimes \mathbb{Q}$ is the orthogonal compliment of $X_q$ in $V'$, we see that $L_q$ has rank $2$. We orient $L_q$ such that $L_q\oplus \langle X_q\rangle$ has the same orientation as $L'$ given by the basis $(\begin{smallmatrix} -1 & \\ & 1 \end{smallmatrix}),(\begin{smallmatrix} & 1 \\ & \end{smallmatrix}),(\begin{smallmatrix} & \\ 1 & \end{smallmatrix})$. 
\begin{lemma}\label{lem:lqt}
    Suppose that $q$ is positive definite and primitive. Then with the restriction of $Q$ the quadratic lattice $L_q$ has discriminant $d$ and $[L_q] = -2[q]$. 
\end{lemma}
\begin{proof}

Since $(X_q,X_q)=2d$ we see that $L_q\oplus \mathbb{Z} X_q$ equals the kernel of the map
\[
    L'\to \mathbb{Z} / d \mathbb{Z}\hspace{1cm} X = \begin{pmatrix} - x & y \\ z & x\end{pmatrix}\mapsto \frac{1}{2}(X,X_q) = bx + a y - cz.
\]
The map is surjective since $(a,b,c)=1$, hence $L_q\oplus \mathbb{Z} X_q$ has index $-d$ in $L'$. Since $L'$ has (unsigned) discriminant $-2$, and $(X_q,X_q)=2d$ it follows that $L_q$ has discriminant $d$. 

To see that $[L_q]=-2[q]$ we show that both sides can be represented by the lattice
\[
    \mathfrak{L} = \mathbb{Z}^3 / (\mathbb{Z}(c, b,a))
\]
with quadratic form given by
\[
    Q_{\mathfrak{L}}(u,v,w) = a^2 u^2 + acv^2 + c^2 w^2 - abuv +(b^2 - 2ac)uw - bc vw,
\]
and oriented such that $(1,0,0),(0,1,0)$ is a positively oriented basis for $\mathfrak{L}\otimes \mathbb{R}$. 

We start with $L_q$. We prove that the matrices $(\begin{smallmatrix} -a & -b \\ & a\end{smallmatrix})$, $(\begin{smallmatrix}  & c \\ a & \end{smallmatrix})$, and $(\begin{smallmatrix} c &  \\ -b & -c\end{smallmatrix})$ generate $L_q$. Let $X=(\begin{smallmatrix} -x & y \\ z & x\end{smallmatrix})\in L_q$. From $(X,X_q)=0$ we have $bx=-ay+cz$ hence since $(a,b,c)=1$ there exists integers $s$ and $t$ such that $x= sa+tc$. This gives 
$a(y+sb)=c(z-tb)$ so
\[
    y = -sb + k \frac{c}{(a,c)} \hspace{1cm} z = tb + k\frac{a}{(a,c)}
\]
for some integer $k$. Then 
\[
    X = s \begin{pmatrix} -a & -b \\ & a \end{pmatrix} +\frac{k}{(a,c)}\begin{pmatrix}  & c \\ a & \end{pmatrix} - t \begin{pmatrix}  c & \\ -b & -c \end{pmatrix} 
\]
We would have been free to replace $(s,t)$ by $(s',t')= (s+\ell \frac{c}{(a,c)},t-\ell \frac{a}{(a,c)})$, which would replace $k$ by $k' = k-b\ell$, so we just have to choose $\ell\equiv k b^{-1} \mod (a,c)$ to ensure that $\frac{k'}{(a,c)}\in \mathbb{Z}$.

For these three generators of $L_q$ the $\mathbb{Z}$-module of relations is free of rank $1$. It is clear that $(c,b,a)\in \mathbb{Z}^3$ belongs to this module, and since $(a,b,c)=1$ it follows that $(c,b,a)$ is a generator. Hence we obtain the desired presentation for $L_q$.

Now observe that a direct calculation gives the ternary quadratic form
\begin{align*}
    (u,v,w)&\mapsto Q\left(u \begin{pmatrix} -a & -b \\ & a \end{pmatrix} + v \begin{pmatrix} & c \\ a &  \end{pmatrix} + w \begin{pmatrix} c & \\ -b & -c \end{pmatrix}\right)
    \\&=-\det \begin{pmatrix} -au + cw & -bu + cv \\ av -wb & au - cw \end{pmatrix} \\
    &= a^2u^2 + ac v^2 + c^2w^2 - abuv +(b^2 - 2ac)uw - bc vw,
\end{align*}
and that $(\begin{smallmatrix} -a & -b \\ & a \end{smallmatrix}), (\begin{smallmatrix} & c \\ a &  \end{smallmatrix})$ is a positively oriented basis of $L_q\otimes \mathbb{R}$ since
\[
    \det \begin{pmatrix} a & 0 & b \\  -b & c& -2c \\ 0 & a & 2a \end{pmatrix} = -ad > 0.
\]
Thus we have identified $L_q$ with $\mathfrak{L}$.

On the other hand, $-[q]$ corresponds to the ideal $\mathfrak{a} = \mathbb{Z}a+\mathbb{Z}\omega$, $\omega=\frac{-b+\sqrt{d}}{2}$ in $\mathcal{O}_d$ with quadratic form given by $\alpha\mapsto a^{-1} \Nm_{\mathbb{Q}(\sqrt{d}) / \mathbb{Q}}(\alpha)$. Hence $-2[q]$ corresponds to $\mathfrak{a}^2$ with the quadratic form  $ \alpha\mapsto a^{-2} \Nm_{\mathbb{Q}(\sqrt{d}) / \mathbb{Q}}(\alpha)$. We have $\mathfrak{a}^2 = \mathbb{Z}a^2 +\mathbb{Z}a\omega+ \mathbb{Z}\omega^2$, and the only relation between these generators is $ca^2+ba\omega+a\omega^2=0$ coming from the minimal polynomial $aX^2+bX+c$ of $\omega/a$. A direct calculation gives
\[
    \Nm_{\mathbb{Q}(\sqrt{d}) / \mathbb{Q}}(u a^2 +va\omega + w\omega^2) = a^4u^2 + a^3c v^2 + a^2c^2w^2 - a^3buv + a^2(b^2 - 2ac)uw - a^2bc vw.
\]
Also $a^2,a\omega$ is a positively oriented basis of $\mathfrak{a}\otimes \mathbb{R}$. This gives the desired presentation for $-2[q]$.
\end{proof}

\begin{lemma}\label{lem:lqq}
Suppose that $q$ is positive definite, and let $g$ be its content. Then $q$ is associated with $X\in(L')^2$ if and only if $X$ spans a subgroup of index $g$ in $L_q$ and $X$ is a positively oriented basis of $L_q\otimes \mathbb{R}$.
\end{lemma}
\begin{proof}
It is clear from the definition that the map $X\mapsto q$ is alternating, hence it induces a linear map
\[
    \Lambda^2 L'\to \mathcal{Q},
\]
where $\mathcal{Q}$ denotes the group of binary quadratic forms. This is in fact a bijection, since the basis elements
\[
    \begin{pmatrix} -1 & \\ & 1 \end{pmatrix}\wedge \begin{pmatrix} & \\  1 & \end{pmatrix} \hspace{1cm} 
    \begin{pmatrix} & \\  1 & \end{pmatrix}\wedge \begin{pmatrix}  & 1 \\ &  \end{pmatrix} \hspace{1cm} 
    \begin{pmatrix} - 1 & \\  & 1 \end{pmatrix}\wedge \begin{pmatrix}  & 1 \\ &  \end{pmatrix} \\
\]
are mapped to the binary quadratic forms $x^2$, $xy$, and $y^2$.

We obtain a bijection
\[
    \left\{ \substack{\text{oriented subgroups $A\subset L'$} \\ \text{such that $L' /A\cong \mathbb{Z}$}}\right\} \cong \{\text{primitive forms $q\in \mathcal{Q}$}\},
\]
which sends $A$ to the binary quadratic form associated with any positively oriented basis of $A$. We claim that $L_q$ corresponds to $q_0=\frac{1}{g}q$ under the above bijection. Let $q_0(x,y)=a_0x^2+b_0xy+c_0y^2$, and let $A\subset L'$ be the span of $X=((\begin{smallmatrix}-a_0 & -b_0 \\ & a_0\end{smallmatrix}),(\begin{smallmatrix} & c_0 \\ a_0 & \end{smallmatrix}))$. An immediate calculation shows that the form associated with $X$ is $a_0q_0$. On the other hand
\[
    [L_q:A] = [L_q\oplus\mathbb{Z}X_q:A\oplus \mathbb{Z}X_q] = \frac{[L':A\oplus \mathbb{Z}X_q]}{[L':L_q\oplus \mathbb{Z}X_q]} = \frac{-a_0d_0}{-d_0} = a_0,
\]
where $d_0 = d /g$, by the previous proof. Writing $X=Yh$, where $Y$ is a positively oriented basis of $L_q$ and $h\in M_2(\mathbb{Z})$, we infer that $\det(h)=a_0$. It follows by Lemma \ref{lem:qii}(iv) that $q_0$ is the form associated with $Y$, which proves the claim.

Now consider any pair of vectors $X\in (L')^2$. We can write $X=Yh$, where $L' / \Span Y\cong \mathbb{Z}$ and $h\in M_2(\mathbb{Z})$ with $\det(h)>0$. By Lemma \ref{lem:qii}(iv) we see that $q$ is associated with $X$ if and only if $q_0$ is associated with $Y$ and $g=\det(h)$. Since $[\Span Y:\Span X]=\det(h)$, the claim implies that $q$ is associated with $X$ if and only if $Y$ is a positively oriented basis of $L_q$ and $[L_q:\Span X]=g$. If $X$ spans a subgroup of $L_q$, then automatically $Y$ must be a basis of $L_q$ with the same orientation as $X$. Therefore this proves the lemma.
\end{proof}
\begin{proof}[Proof of Proposition \ref{prp:mtq}]
By Lemma \ref{lem:lqq} $m(t,q)$ is equal to the number of positively oriented pairs of vectors $X\in L_q^2$ such that $[L_q:\Span X]=g$ and $Q(X)=T$. However, if $X\in L_q^2$, then the condition that $Q(X)=T$ implies that $[L_q:\Span X]=g$, since $L_q$ has discriminant $d /g$, and $t$ has discriminant $d$. Hence $m(t,q)$ is equal to the number of positively oriented pairs of vectors $X\in L_q^2$ such that $Q(X)=T$, or equivalently such that
\[
    Q(X_1x+X_2y) = t(x,y).
\]
By Lemma \ref{lem:lqt} this is equal to $r_{-2[q_0]}(t)$.
\end{proof}

%--------------------------------------------------------

\subsection{Computations}

For a closed oriented curve $c'$ in $M'$ and an oriented point $x\in M'$ let $w(x;c')$ be the intersection number of an oriented curve from $x$ to $\infty$ and $c'$, where intersections on $x$ are counted with half-multiplicity. When $c'$ is the boundary of a map $s'$ then $w(x;c')$ equals the intersection number of $s'$ and $x$, where similarly intersections on $x$ are counted with half-multiplicity. Hence we have
\[
    \iota_{L'}(T;s')_0 = \sum_{q\in \mathcal{Q}_d / \Gamma'} m(T,q) w(x(q);c').
\]
The following lemma explains how to compute  $w(x;c')$.

\begin{lemma}\label{lem:wna}
    Let $x\in M'$ be a positively oriented point and let $z\in \mathcal{F}$ be the unique lift of $x$. Let $c'$ be a closed geodesic associated with an indefinite binary quadratic form $q'$ of discriminant $d'>0$. Let $q_i'(x,y)=a_i'x^2+b_i'xy+c_i'y^2$, $i=1,\ldots,\ell$, be the list of forms in the $\Gamma'$-class of $q'$ for which the associated hyperbolic line in $\mathcal{H}^2$ intersects $\mathcal{F}$. Then
    \begin{equation}\label{wnaf}
        w(x;c') =  \sum_{i = 1}^\ell \sgn(a_i')\frac{\sgn\left(\frac{\sqrt{d'}}{2|a_i'|}-\left|z+\frac{b_i'}{2a_i'}\right|\right)+1}{2}
    \end{equation}
\end{lemma}
\begin{proof}
Let $c_i'\subset \mathcal{H}^2$ be the hyperbolic line corresponding to $q_i'$. The lemma follows as $\ell_{z,\infty}\subset \mathcal{F}$ and the oriented intersection number of $\ell_{z,\infty}$ and $c_i'$ is given by the $i$th term in \eqref{wnaf}.
\end{proof}

\begin{example}
    Let $z_1=\frac{1+\sqrt{-7}}{4}$ and $z_2=\frac{1+\sqrt{-23}}{12}$ and let $x_1,x_2\in M'$ be their images. Let $c'$ be the geodesic associated with $q'(x,y)=2x^2+2xy-y^2$. Then $w(x_1;c')=1$ and $w(x_2;c')=2$. This is depicted in Figure \ref{wne}.
\begin{figure}[ht]
        \centering
    \includegraphics{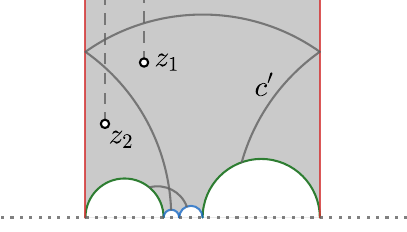}
        \caption{The segments of $c'$ which intersect  $\ell_{z_1,\infty}$ and $\ell_{z_2,\infty}$ are oriented counter clockwise, hence the intersection numbers are $+1$ and $+2$.}
        \label{wne}
\end{figure}
\end{example}

Since $m(T,q)$ only depends on the $\PSL_2(\mathbb{Z})$-orbit of $T$ and $q$, we can collect terms obtaining
\begin{equation}\label{cte}
    \iota_{L'}(T;s')_0 = \sum_{\mathfrak{C}\in \mathcal{C}_d} m(T,\mathfrak{C}) \sum_{q\in \mathfrak{C} / \Gamma'} w(x(q);c')-w(x(-\tilde q);c').
\end{equation}
Given a representative $q\in \mathfrak{C}$ with $z_q\in \mathcal{F}_1$, we can compute a set of representatives for $\mathfrak{C} / \Gamma'$ and the associated points in $\mathcal{F}$ by translating $z_q$ and $-1/z_{-\tilde q}$ by representatives of $\PSL_2(\mathbb{Z}) / \Gamma'$.

By \eqref{iip} we have
\begin{equation}\label{eq:wlr}
    \iota_{L}(T;s)_0 = \sum_{n,m\in \mathbb{Z}} \iota_{L'}\left(T - \begin{pmatrix} n^2 & nm \\ nm & m^2 \end{pmatrix} ;s'\right)_0.
\end{equation}
The matrix $T-(\begin{smallmatrix} n^2 & nm \\ nm & m^2\end{smallmatrix})$ is only positive definite for finitely many $(n,m)$ so the sum is finite. %Using \eqref{eq:wlr} we computed the values of $\iota_L(T;s)$ in Table \ref{tab:nzl}.

\begin{example}\label{exa:ntwe}
For $\det T<23/4$, every $T$ is $\PSL_2(\mathbb{Z})$-equivalent to a diagonal matrix and hence $\iota_{L'}(T;s')_0=0$ and $\iota_L(T;s)_0=0$.

On the other hand, let $T = (\begin{smallmatrix}
        2 & 1/2 \\ 1/2 & 3
\end{smallmatrix})$. A list of representatives $q$ for $\mathcal{C}_{-23}$ with $z_q\in \mathcal{F}_1$ is given in Table \ref{tab:rzm}. 
\begin{table}[ht]
\centering
\begin{tabular}{c c c c}
            $q$ & $z_q$  & $m(T;q)$ \\[.5ex]\hline \\[-2ex]
            $x^2-xy+6y^2$ & $\frac{1+\sqrt{-23}}{2}$ & $0$ \\
            $2x^2-xy+3y^2$ & $\frac{1+\sqrt{-23}}{4}$ &  $0$ \\
            $3x^2-xy+2y^2$ & $\frac{1+\sqrt{-23}}{6}$ &  $2$ 
\end{tabular}
\caption{Representatives for $\mathcal{C}_{-23}$}
\label{tab:rzm}
\end{table}

We see that only the orbit $\mathfrak{C}$ represented by $3x^2-xy+2y^2$ contributes to \eqref{cte}. For this orbit we computed
\[
    \sum_{q\in \mathfrak{C} / \Gamma'} w(x(q);c_i')-w(x(-\tilde q);c_i') = \begin{cases}
        0 & i =1,2 \\
        4 & i =3
    \end{cases}
\]
where $c_1'$, $c_2'$, and $c_3'$ are the three geodesics in Table \ref{tab:ght}. Letting $c'$ be their union we obtain
\[
    \iota_{L'}(T;s')_0 = 8.
\]

Next, we compute $\iota_L(T;s)$ using \eqref{eq:wlr}. For $n,m\in \mathbb{Z}$, the matrix $T'=T-(\begin{smallmatrix}
        n^2 & nm \\ nm & m^2
    \end{smallmatrix})$ is positive definite if and only if $(n,m)\in\{(0,0),(\pm1,0),(0,\pm1),(1,-1),(-1,1)\}$, and so $T'$ equals one of the matrices
    \[
        \begin{pmatrix} 2 & 1 /2 \\ 1 /2 & 3 \end{pmatrix},\;
        \begin{pmatrix} 1 & 1 /2 \\ 1 /2 & 3 \end{pmatrix},\;
        \begin{pmatrix} 2 & 1 /2 \\ 1 /2 & 2 \end{pmatrix},\;
        \begin{pmatrix} 1 & -1 /2 \\ -1 /2 & 2 \end{pmatrix}.
    \]
    The last three matrices have $\det T'<23/4$, and hence $\iota_{L'}(T';s')_0=0$. Thus by \eqref{eq:wlr}
    \[
        \iota_L(T;s)_0 = \iota_{L'}(T;s')_0 = 8.
    \]
\end{example}

\begin{remark}
If $c':C\to M'$ is a geodesic with $[c']\neq0$, let $c_0'$ be a union of horocycles centered at $x_2$, $x_3$, and $x_4$ such that $[c_0']=[c']$. Then there exists a surface $S_0$ and a smooth map $s_0':S_0\to M'$ with boundary $c'-c_0'$. We can choose $c_0'$ such that $w(x(q);c_0')=0$ for all positive binary quadratic forms $q$ of discriminant $d$. Then it follows that
\begin{equation}\label{cte2}
    \iota_{L'}(T;s_0')_0 = \sum_{\mathfrak{C}\in \mathcal{C}_d} m(T,\mathfrak{C}) \sum_{q\in \mathfrak{C} / \Gamma'} w(x(q);c')-w(x(-\tilde q);c').
\end{equation}
Let $s_0$ (resp. $c$) be the composite map $S_0\to M'\to M$ (resp. $C\to M'\to M$). Then $\iota_L(T;s_0)_0$ can also be computed by \eqref{eq:wlr}, and if we choose $c_0'$ close enough to the cusps, then this computes the linking numbers of $c$ as in Remark \ref{rem:sgl}.

In this way, we obtained the data in Table \ref{tab:nzl} (the curve $c$ in that example corresponds to $c_3'$ in Table \ref{tab:ght}).
\end{remark}

%--------------------------------------------------------

\bibliographystyle{plain}
\bibliography{bibliography}

\end{document}